\documentclass[11pt,reqno]{amsart}
\usepackage{amsmath,amsfonts,amsthm,amsgen,mathrsfs,amssymb,cite}

\usepackage{graphicx}
\input xy
\xyoption{all}

\title{GCR and CCR Steinberg algebras}
\author{Lisa Orloff Clark\and Benjamin Steinberg \and Daniel W van Wyk}
\address{Lisa Orloff Clark\\ School of Mathematics and Statistics\\ Victoria University of Wellington \\ PO Box 600\\ Wellington 6140 \\ New Zealand}
\email{lisa.clark@vuw.ac.nz}
\address{Benjamin Steinberg\\
    Department of Mathematics\\
    City College of New York\\
    Convent Avenue at 138th Street\\
    New York, New York 10031\\
    USA}
\email{bsteinberg@ccny.cuny.edu}
\address{Daniel W van Wyk \\ Federal University of Santa Catarina \\ Department of Mathematics \\ Campus Universit\'ario Trindade CEP 88.040-900, Florian\'opolis-SC, Brazil}
\email{dwvanwyk79@gmail.com}
\thanks{The second author thanks the Fulbright Commission for sponsoring his stay at the Universidade Federal de Santa Catarina in Florianopolis, Brazil where much of this work was done}
\date{January 17, 2019}

\newtheorem{thm}{Theorem}[section]
\theoremstyle{plain}
\newtheorem*{thm*}{Theorem}
\newtheorem{prop}[thm]{Proposition}
\newtheorem*{prop*}{Proposition}
\newtheorem{lem}[thm]{Lemma}
\newtheorem{cor}[thm]{Corollary}
{\theoremstyle{definition}
\newtheorem{dfn}[thm]{Definition}}
{\theoremstyle{remark}
\newtheorem{ex}[thm]{Example}}
{\theoremstyle{remark}
\newtheorem*{ex*}{Example}}
{\theoremstyle{definition}
}
{\theoremstyle{remark}
\newtheorem{remark}[thm]{Remark}}
\numberwithin{equation}{section}

\newcommand{\prim}{\mathrm{Prim}}

\newcommand{\G}[0]{\mathscr{G}}
\newcommand{\orb}[0]{\mathcal{O}}



\begin{document}

\begin{abstract}
Kaplansky introduced the notions of CCR and GCR $C^*$-algebras because they have a tractable representation theory.  Many years later, he introduced the notions of CCR and GCR rings.  In this paper we characterize when the algebra of an ample groupoid over a field is CCR and GCR. The results turn out to be exact analogues of the corresponding characterization of locally compact groupoids with CCR and GCR $C^*$-algebras. As a consequence, we classify the CCR and GCR Leavitt path algebras.
\end{abstract}

\maketitle

\section{Introduction}
It is well known that every $C^*$-algebra can be represented as a norm-closed self-adjoint  algebra of operators on a Hilbert space. The structure of commutative $C^*$-algebras is well understood due to a theorem of Gelfand, which shows that every commutative $C^*$-algebra is isomorphic to the continuous functions that vanish at infinity on a locally compact Hausdorff space. The locally compact Hausdorff space is
%set of non-zero *-homomorphisms of the $C^*$-algebra into the complex numbers endowed with the weak*-topology, called the spectrum. The spectrum is
precisely the space of irreducible representations of the $C^*$-algebra, all of which are one-dimensional.

%Moreover, the spectrum is homeomorphic to the set of maximal ideals of the $C^*$-algebras with the Jacobson topology.

Guided by the structure of commutative $C^*$-algebras Kaplansky introduced the classes of CCR and GCR $C^*$-algebras~\cite{Ka51}. Kaplansky defines a $C^*$-algebra to be CCR if its image under any irreducible representation is precisely the compact operators, that is, the norm closure of  all the  finite rank operators. Hence every commutative $C^*$-algebra is CCR and, in the sense of their representation theory, the class of non-commutative CCR $C^*$-algebras is very similar to commutative $C^*$-algebras. A $C^*$-algebra is GCR if the image of every irreducible representation contains the compact operators. These are the two most accessible classes of $C^*$-algebras in term of their representation theories. In fact, there is evidence that strongly suggests a parametric description of the irreducible representations of non-GCR $C^*$-algebras is not possible (see for example the discussion at the end of Section 3.5 in~\cite{Arv76}).

Subsequent to Kaplansky's introduction, CCR and GCR characterizations have been found for $C^*$-algebras associated to directed graphs~\cite{Eph04} and  to transformation groups~\cite{Go73,Wil81}, which have been generalized to groupoids~\cite{Cl07,vW18,vW18a}. %In the case of groupoids $C^*$-algebras, the topology of the orbit space and the isotropy groups determine if the associated $C^*$-algebra is CCR or GCR.
Here we characterize CCR and GCR Steinberg algebras associated to ample groupoids in an analogous manner to the characterizations for groupoid $C^*$-algebras. Given an ample groupoid, we show that, for the Steinberg algebra to be CCR (respectively, GCR), it is necessary and sufficient for the groupoid to have a $T_1$ (respectively, $T_0$) orbit space and for the isotropy group algebras to be CCR (respectively, GCR).  Let us first say a word about these notions in the algebraic setting.

In~\cite{K89} Kaplansky defines a ring as CCR if every primitive quotient is simple and has a minimal left ideal. He defines a ring to be GCR if every primitive quotient has a minimal left ideal.  In Section~\ref{sec:Rings} we develop properties of CCR and GCR rings and algebras. In particular, in the case of an algebra over an algebraically closed field such that the algebra's dimension is less than the cardinality of the field, we reformulate Kaplansky's definitions more concretely in terms of finite rank operators so that it mirrors the $C^*$-algebraic versions (see Proposition~\ref{p:algclosecase}).  This applies in particular to algebras of countable dimension over the complex numbers.   The reason that we must impose these kinds of restrictions is that an operator acting on a complex vector space $V$ may have empty spectrum if the dimension of $V$ is not countable. Hence Schur's lemma (see Lemma~\ref{l:schur.dim}) does not apply in this setting to give that the commutant consists of  just the scalars, as would happen in the $C^*$-algebraic context.

In Section~\ref{sec:Morita} we show that the properties of being CCR and GCR are Morita-invariant for rings. This section is not used again in this paper, and may be omitted.

Since our CCR and GCR characterizations in part require the isotropy groups to be CCR or GCR,  Section~\ref{sec:CCRgrouprings} considers certain groups that have CCR group rings. We use a variation of Clifford's theorem (Theorem~\ref{t:clifford}) to show that the group ring of a virtually abelian group of cardinality less than that of an algebraically closed field is CCR over this field (Corollary~\ref{c:CCR.group}). As a special case, if the field is the complex numbers, then it follows that every countable virtually abelian group has a CCR group ring (Corollary~\ref{c:CG.complex}). We end this section by stating a theorem of Roseblade (Theorem~\ref{t:Rosebade}), which says that the group ring of a virtually polycyclic group over an algebraic extension of a finite field is CCR.

In Section~\ref{sec:steinbergalgs} we establish the necessary background and notation on groupoids, their associated Steinberg algebras and the representation theory of Steinberg algebras. We also prove some preliminary results that are used later on. Proposition~\ref{p:locally.closed}  proves parts of the Ramsay-Mackey-Glimm dichotomy in the special case of ample groupoids without the standard assumption that the groupoid is Hausdorff. Proposition~\ref{p:discrete.groupoid} gives a CCR and GCR characterization for discrete transitive groupoids. We also review a sheaf theoretic approach to simple modules, introduced by the second author~\cite{St14}, and prove in Proposition~\ref{p:simple.sheaf} that the action of a Steinberg algebra on a simple module factors through a quotient corresponding to reduction to an orbit closure.

Sections~\ref{sec:CCRcharacterization} and~\ref{sec:GCRcharacterization} contain our main results. These are Theorem~\ref{t:CCRthm} and Theorem~\ref{t:GCRthm}, which gives CCR and GCR characterization, respectively,  for Steinberg algebras.

In Section~\ref{sec:higerRankgraphs} we apply our results to finitely aligned higher rank graph algebras. In Propositions~\ref{p:condM} and~\ref{p:condN} we give necessary and sufficient conditions for the orbit space of the boundary path groupoid to be $T_1$ and $T_0$, respectively. These propositions are analogous of, and generalize, Ephrem's CCR and GCR conditions (cf.~\cite{Eph04}). Then in  Corollary~\ref{c:ccrgraph} and~\ref{c:gcrgraph} we give CCR and GCR characterizations for finitely aligned higher rank graphs.  In particular, these results classify the Leavitt path algebras~\cite{LeavittBook} that are CCR and GCR.
Finally in Example~\ref{ex:graph} we give an example of a nice directed graph whose associated boundary path groupoid has non-discrete $T_1$ orbit space and hence the associated Leavitt path algebra is CCR and GCR.

\section{CCR and GCR rings} \label{sec:Rings}
Inspired by the theory of $C^*$-algebras, Kaplansky~\cite{K89} defined the notions of CCR and GCR rings.  In what follows we shall use left modules and hence need to speak of left primitive rings, etc.  Since groupoid algebras have an involution, these distinctions are irrelevant for us and so we mostly omit the term left.  We do not require rings to be unital and use the term `ideal' to mean two-sided ideal.

If $R$ is a ring and $M$ is an $R$-module, then $M$ is \emph{simple} if $RM\neq \{0\}$ and $M$ has no proper submodules except $M$ and $\{0\}$.  One has that $M$ is simple if and only if $M\neq 0$ and $Rm=M$ for all $0\neq m\in M$.   The ring $R$ is (left) \emph{primitive} if it has  a faithful simple (left) module. If $R$ is a unital ring, or more generally a ring with local units, then $R$ always has simple modules.
 If $M$ is a simple $R$-module, then by Schur's lemma $D=\mathrm{End}_R(M)$ is a division ring (or skew field) and $R$ acts on $M$ by $D$-linear maps.  By Jacobson's density theorem, it acts densely in the sense that if $m_1,\ldots, m_k\in M$ are $D$-linearly independent and $m_1',\ldots, m_k'\in M$, then there exists $r\in R$ with $rm_i=m_i'$ for $i=1,\ldots, k$.

Recall that  a ring $R$ is simple if $R^2\neq 0$ and $R$ has no proper non-zero ideals.  The \emph{socle} $\mathrm{soc}(R)$ of a ring $R$ is the left ideal generated by the set of minimal left ideals in $R$ (hence it is $0$ if $R$ has no minimal left ideals).  Note that $\mathrm{soc}(R)$ is a two-sided ideal and, if $R$ is primitive (or more generally semiprime), is generated as a right ideal by the set of minimal right ideals in $R$.  We remark that a minimal left ideal is the same thing as a simple submodule of the regular module $R$.

\begin{dfn}[Kaplansky]
A ring $R$ is (left) \emph{CCR} if every (left) primitive quotient ring of $R$ is simple and has a minimal left ideal.  We say that $R$ is (left) \emph{GCR} if every (left) primitive quotient ring of $R$ has  a minimal left ideal.
\end{dfn}

The following is an immediate consequence of the definitions.

\begin{prop}\label{p:quotients.ccr}
If the ring $R$ is CCR (respectively, GCR), then so is $R/I$ for any ideal $I$ of $R$.
\end{prop}

A crucial point is that if a primitive ring $R$ has a minimal left ideal $L$, then $L$ is the unique faithful simple $R$-module up to isomorphism.

\begin{prop}\label{p:uniquesimple}
Let $R$ be a primitive ring with a minimal left ideal $L$.  Then every faithful simple $R$-module is isomorphic to $L$.
\end{prop}
\begin{proof}
Let $M$ be a faithful simple $R$-module.  Then $L$ does not annihilate $M$ and so there exists $m\in M$ with $Lm\neq 0$.  Since $M$ is simple, we must have $Lm=M$.  Also the mapping $\psi\colon L\to M$ defined by $\psi(r)=rm$ is a surjective $R$-module homomorphism and hence an isomorphism as $L$ is minimal.
\end{proof}

A consequence of Proposition~\ref{p:uniquesimple} and the definitions is the following.

\begin{prop}\label{p:gcr.spectrum}
Let $R$ be a GCR ring.  Then two simple modules are isomorphic if and only if they have the same annihilator.  That is, the primitive ideal spectrum of $R$ is in bijection with the set of isomorphism classes of simple $R$-modules.
\end{prop}

By~\cite[Corollary~5.33]{Bre14} a primitive ring $R$, with faithful simple module $M$, contains a minimal left ideal if and only if it contains $r\in R$ that acts on $M$ as a non-zero finite rank operator (with respect to the $D$-vector space structure where $D=\mathrm{End}_R(M)$).  Moreover, $\mathrm{soc}(R)$ consists of the set of all elements of $R$ that act on $M$ as finite rank operators. In~\cite[Theorem~8.1]{CohnBA3}, it is shown that a ring $R$ is simple with a minimal left ideal if and only if it is isomorphic to a dense ring of finite rank operators over a division ring, which can moreover be taken to be the endomorphism ring of one of its minimal left ideals (viewed as its unique, up-to-isomorphism, faithful simple left $R$-module).

These observations allow us reformulate the notions of CCR and GCR rings in a concrete fashion that more closely mirrors the operator algebra setting. In the operator setting Kaplansky defines a $C^*$-algebra to be CCR if its image under any irreducible representation is precisely the compact operators~\cite{Ka51}, that is, the norm closure of the span of all finite rank operators. Kaplansky defines a C$^*$-algebra to be GCR if it has a composition series of closed two-sided ideals such that quotients of successive terms are CCR. However, in~\cite[Theorem 1]{Gl73} Glimm shows this GCR definition is equivalent to having the image of every irreducible representation contain a compact operator.\footnote{
	In fact, in the case of $C^*$-algebras if the image of an irreducible representation contains a non-zero compact operator then it contains all compact operators.}

\begin{prop}\label{p:ccrgcrnice}
Let $R$ be a ring.
\begin{enumerate}
  \item The ring $R$ is CCR if and only if its action on any simple $R$-module $M$ is by finite rank operators over the division ring $\mathrm{End}_R(M)$.
  \item The ring $R$ is GCR if and only if its action on any simple $R$-module $M$ contains a non-zero finite rank operator over the division ring $\mathrm{End}_R(M)$.
\end{enumerate}
\end{prop}
\begin{proof}
If $M$ is a simple $R$-module, then $\overline{R}=R/\mathrm{ann}(M)$ is primitive and  $M$ is a faithful $\overline{R}$-module.  Thus $\overline{R}$ has an element acting as a non-zero finite rank operator if and only if it has a minimal left ideal by~\cite[Corollary~5.33]{Bre14}. This proves the second statement.

Suppose now that $R$ is CCR and let us retain the above notation. Then $\mathrm{soc}(\overline{R})$ is the collection of elements of $\overline{R}$ acting as finite rank operators (see~\cite[Theorem~5.30]{Bre14}), and it is a non-zero ideal (as $\overline{R}$ has a minimal left ideal).  So if $\overline{R}$ is simple, then $\overline{R}=\mathrm{soc}(\overline{R})$ and hence $R$ acts on $M$ by finite rank operators. On the other hand, since any dense ring of finite operators over a division ring is simple with a minimal left ideal by~\cite[Theorem~8.1]{CohnBA3}, we obtain the converse to the first statement.
\end{proof}

If $\Bbbk$ is a field, then a representation of a $\Bbbk$-algebra $A$ is a $\Bbbk$-algebra homomorphism $\rho\colon A\to \mathrm{End}_{\Bbbk}(V)$ for some $\Bbbk$-vector space $V$.   The representation is \emph{irreducible} if $V$ is a simple $A$-module with respect to the natural module structure $av=\rho(a)V$.

Schur's lemma has a well-known strengthening over algebraically closed fields for algebras of sufficiently small dimension.

\begin{lem}\label{l:schur.dim}
Let $A$ be an algebra over an algebraically closed field $\Bbbk$  and suppose that the dimension of $A$ is less than the cardinality of $\Bbbk$ and let $V$ be a simple $A$-module. Then $\mathrm{End}_A(V)=\Bbbk$.
\end{lem}
\begin{proof}
Let $D=\mathrm{End}_A(V)$.  Note that $D$ is a division algebra over $\Bbbk$ by Schur's lemma.  Let us show that its dimension is less than the cardinality of $\Bbbk$.
First note that if $v\neq 0$ belongs to $V$, then $Av=V$ and so $V$ is a quotient of $A$ as a vector space and hence has dimension less than the cardinality of $\Bbbk$.  Also note that $f\mapsto f(v)$ is an injection from $D$ to $V$ as $f(av)=af(v)$ for all $a\in A$ and $Av=V$.  Thus $D$ has dimension less than the cardinality of $\Bbbk$.

So it suffices to show that if $D$ is a division algebra over an algebraically closed field $\Bbbk$ of dimension less than the cardinality of $\Bbbk$, then $D=\Bbbk$.  Suppose that $d\in D\setminus \Bbbk$.  Then $\Bbbk(d)\subseteq D$ is an extension field of $\Bbbk$ of dimension strictly less than the cardinality of $\Bbbk$.   Consider the family of elements $X=\{(d-\lambda)^{-1}\mid \lambda\in \Bbbk\}\subseteq \Bbbk(d)$.  Then $X$  has the same cardinality as $\Bbbk$.  Thus $X$ is linearly dependent over $\Bbbk$ and so there are $\mu_i,\lambda_i\in \Bbbk$, not all $\mu_i$  zero, such that
\[\sum_{i=1}^n \frac{\mu_i}{d-\lambda_i}=0.\]  Clearing denominators, we get a non-zero polynomial $p(x)\in \Bbbk [x]$ with $p(d)=0$.  Thus $\Bbbk(d)$ is an algebraic extension of the algebraically closed field $\Bbbk$ and so $\Bbbk(d)=\Bbbk$, contradicting $d\notin \Bbbk$.
\end{proof}

The following is now an immediate consequence of the lemma.

\begin{prop}\label{p:algclosecase}
Let $A$ be an algebra over an algebraically closed field $\Bbbk$ of dimension less than the cardinality of $\Bbbk$.
\begin{enumerate}
  \item $A$ is CCR if and only if all its irreducible representations over $\Bbbk$  are by finite rank operators.
  \item $A$ is GCR if and only if all its irreducible representations over $\Bbbk$  contain a non-zero finite rank operator in the image.
\end{enumerate}
\end{prop}

We typically apply Proposition~\ref{p:algclosecase} when $\Bbbk=\mathbb C$ and $A$ is of countable dimension.  Note that Proposition~\ref{p:algclosecase} is false without the dimension assumption.  For example, the field $\mathbb C(x)$ of rational functions is clearly CCR but its only irreducible representation over $\mathbb C$ is the regular representation and this representation is by uncountable rank operators over $\mathbb C$.  However, it is by rank $1$ operators over $\mathbb C(x)$, which is the endomorphism algebra.

A ring $R$ is said to have \emph{local units} if there is a collection $E\subseteq R$ of idempotents, called a \emph{set of local units} for $R$, such that if $F\subseteq R$ is finite, then $F\subseteq eRe$ for some $e\in E$.
A module $M$ over a ring $R$ with local units is called \emph{unitary} if $RM=M$.  If $E$ is  a set of local units for $R$, then $M$ is unitary if and only if, for all $m\in M$, there exists $e\in E$ with $em=m$.  Every simple module is unitary.

It will turn out to be useful to prove that the property of being CCR or GCR is stable under taking matrices.  If $R$ is a ring and $I$ is an index set, then $M_I(R)$ is the ring of all $|I|\times |I|$-matrices over $R$ (i.e., functions $A\colon I\times I\to R$) with finitely many non-zero entries with respect to usual matrix addition and multiplication.  Note that if $W$ is an $R$-module and $W^{(I)}$ denotes a direct sum of copies of $W$ indexed by $R$, then $W^{(I)}$ is naturally a left $M_I(R)$-module via matrix multiplication (where we view $W^{(I)}$ as consisting of column vectors).  Suppose that $R$ is unital.  If $J\subseteq I$ is finite, let $1_J$ be the matrix which has $1$ in the diagonal entries corresponding to $j\in J$ and zeroes in all other entries;  then the $1_J$ form a set of local units for $M_I(R)$.   The following proposition is undoubtedly well known.

\begin{prop}\label{p:morita.matrix}
Let $R$ be a unital ring and $I$ an index set. The functor taking $W$ to $W^{(I)}$ is an equivalence between the categories of unitary $R$-modules and unitary $M_I(R)$-modules.
\end{prop}
\begin{proof}
Fix $i_0\in I$ and let $e=1_{\{i_0\}}$. Then $e$ is idempotent and $eM_I(R)e\cong R$ and so we shall identify these rings.  We shall show that $V\mapsto eV$ is  quasi-inverse to $W\mapsto W^{(I)}$.

Clearly $eW^{(I)}\cong W$ via projection to the $i_0$-factor.  Conversely, let $V$ be a unitary $M_I(R)$-module and put $W=eV$, viewed as an $R$-module.  We show that $V\cong W^{(I)}$.  For $i\in I$, let $A^{(i)}$ be the matrix with $A^{(i)}_{i_0i}=1$ and all other positions $0$ and note that $eA^{(i)}=A^{(i)}$.  Let $B^{(i)}$ be the transpose of $A^{(i)}$ (so it has $1$ in position $ii_0$).  Define $\Phi\colon V\to W^{(I)}$ by $\Phi(v)_i = A^{(i)}v\in eV=W$.

To see that $\Phi$ is well defined, we use that $V$ is unitary.  Then there is a finite subset $J\subseteq I$ with $1_Jv=v$.  If $i\notin J$, then $A^{(i)}1_J=0$ and hence $A^{(i)}v=  A^{(i)}1_Jv=0$.  Thus $\Phi(v)\in W^{(I)}$.

Let us check that $\Phi$ is a module homomorphism.  Let $C\in M_I(R)$ and $v\in V$.  Assume that $J\subseteq I$ is finite with $1_Jv=v$ and note that \[1_J = \sum_{j\in J}B^{(j)}A^{(j)}.\]   Then we have
\begin{align*}
\Phi(Cv)_i &= A^{(i)}Cv = A^{(i)}C1_Jv = A^{(i)}C\left(\sum_{j\in J}B^{(j)}A^{(j)}v\right) \\ &= \sum_{j\in J}C_{ij}A^{(j)}v = \sum_{j\in J}C_{ij}\Phi(v)_j = (C\Phi(v))_i
\end{align*}
using that all the non-zero entries of $\Phi(v)$ belong to positions in $J$ by the observation of the preceding paragraph.

It remains to show that $\Phi$ is bijective.  Suppose that $\Phi(v)=0$ and that $1_Jv=v$ with $J\subseteq I$ finite.  Then \[v=1_Jv =  \sum_{j\in J}B^{(j)}A^{(j)}v= \sum_{j\in J}B^{(j)}\Phi(v)_j=0\] and so $\Phi$ is injective.  If $w=(w_i)_{i\in I}\in W^{(I)}$ and $J\subseteq I$ is the finite subset of indices $i$ such that $w_i\neq 0$, let $v=\sum_{j\in J}B^{(j)}w_j\in V$.  Then $\Phi(v)_i=0$ unless  $i\in J$, in which case it is $A^{(i)}\sum_{j\in J}B^{(j)}w_j=A^{(i)}B^{(i)}w_i=ew_i=w_i$.  Thus $\Phi(v)=w$.  This completes the proof.
\end{proof}

As a corollary, we see that CCR and GCR are stable under matrix amplifications.

\begin{cor}\label{c:matrix.amp}
Let $R$ be a unital ring and $I$ an index set.
\begin{enumerate}
\item $R$ is CCR if and only if $M_I(R)$ is CCR.
\item $R$ is GCR if and only if $M_I(R)$ is GCR.
\end{enumerate}
\end{cor}
\begin{proof}
We retain the notation of the proof of Proposition~\ref{p:morita.matrix}, in particular that of the idempotent $e$.  By Proposition~\ref{p:morita.matrix}, the simple
$M_I(R)$-modules are, up to isomorphism, those of the form $W^{(I)}$ with $W$ a simple $R$-module.  Moreover, if $D=\mathrm{End}_R(W)$, then $D\cong \mathrm{End}_{M_I(R)}(W^{(I)})$ by Proposition~\ref{p:morita.matrix}.  The action of $D$ on $W^{(I)}$ is diagonal: $d(w_i)_{i\in I} = (dw_i)_{i\in I}$.

If we assume that $R$ is CCR, then the action of $R$ on $W$ is by finite rank operators over $D$ and hence the action of a matrix $A\in M_I(R)$ is by finite rank operators on $W^{(I)}$ as the matrix has only finitely many non-zero entries and each entry is a finite rank operator on $W$ (more precisely, $AW^{(I)}\subseteq \bigoplus_{i\in I}\sum_{j\in I}A_{ij}W$ which has finitely many non-zero summands each of finite $D$-dimension).  Conversely, if $M_I(R)$ is CCR, then the rank of $r\in R$ on $W$ is the same as the rank of $re$ on $W^{(I)}$ and hence is finite. This yields the first item.

If $R$ is GCR and $r\in R$ acts as a non-zero finite rank operator on $W$, then $re$ acts as a finite rank operator on $W^{(I)}$ of the same rank and so $M_I(R)$ is GCR.  If $M_I(R)$ is GCR and the matrix $C$ acts as a non-zero finite rank operator on $W^{(I)}$, then we can find a coefficient $C_{ij}$ with non-zero action on $W$. Then $A^{(i)}CB^{(j)}=C_{ij}e$ is a finite rank, non-zero operator   on $W^{(I)}$ and hence on $eW^{(I)}\cong W$, i.e., $C_{ij}$ acts as a non-zero finite rank operator on $W$.  This establishes the second item.
\end{proof}

The following lemma shall be used later to in order to relate the CCR/GCR properties for a ring to those of its ideals.

\begin{lem}\label{l:restrict.ideals}
Let $R$ be a ring and $I$ an ideal of $R$.
\begin{enumerate}
\item Let $M$ be a simple $R$-module such that $IM\neq 0$.  Then $M$ is a simple $I$-module and $\mathrm{End}_R(M)=\mathrm{End}_I(M)$.
\item Suppose that $M$ is a simple $I$-module. Then $M$ has a unique $R$-module structure extending the $I$-module structure and, moreover, $M$ is simple as an $R$-module.
\end{enumerate}
\end{lem}
\begin{proof}
For the first item, let $0\neq m\in M$.  Then $Im$ is an $R$-submodule since $RIm\subseteq Im$. So either $Im=0$ or $Im=M$.  If $Im=0$, then since $M=Rm$, we have that $IM=IRm\subseteq Im=0$, a contradiction.  We conclude that $M$ is a simple $I$-module.  Clearly, $\mathrm{End}_R(M)\leq \mathrm{End}_I(M)$.  Let $\Phi\in \mathrm{End}_I(M)$, $m\in M$ and $r\in R$.  Since $IM=M$, we can write $m =\sum r_im_i$ with $r_i\in I$.  Then, as each $rr_i\in I$, we have
\begin{align*}
\Phi(rm) &= \Phi\left(\sum rr_im_i\right) = \sum rr_i\Phi(m_i) = r\sum r_i\Phi(m_i)\\
& = r\Phi\left(\sum r_im_i\right)=r\Phi(m)
\end{align*}
 and so $\Phi\in \mathrm{End}_R(M)$.

For the second item, note that since $M=IM$, if $m\in M$, then $m=\sum_{i=1}^k r_im_i$ with the $r_i\in I$ and $m_i\in M$.  If there is an $R$-module structure extending the $I$-module structure on $M$, we must have that
\begin{equation}\label{eq:def.ract}
rm=\sum_{i=1}^k(rr_i)m_i
\end{equation}
 for $r\in R$, where we note that $rr_i\in I$.    So let us use \eqref{eq:def.ract} as a definition for $rm$.  To show that the $R$-action is well defined, it suffices to show that if $\sum_{i=1}^nr_im_i=0$, with $r_i\in I$ and $m_i\in M$, then $\sum_{i=1}^k(rr_i)m_i=0$ for all $r\in R$.  Suppose that $\sum_{i=1}^k(rr_i)m_i\neq 0$, let's call this element $n$.  Then, since $M$ is simple over $I$, there exists $s\in I$ with $sn=n$.  Thus we have
 \[0\neq n=sn = s\sum_{i=1}^k (rr_i)m_i = \sum_{i=1}^k(sr)r_im_i=(sr)\sum_{i=1}^kr_im_i =0\] which is a contradiction.  Therefore, \eqref{eq:def.ract} gives a valid definition of $rm$.  It is straightforward to check that with this definition $M$ is an $R$-module.
  Moreover, if $r\in I$ and $m\in M$ with $m=\sum_{i=1}^k r_im_i$ with $r_i\in I$ and $m_i\in M$, then \[\sum_{i=1}^k(rr_i)m = \sum_{i=1}^k r(r_im) = r\sum_{i=1}^kr_im_i = rm\] and so the $R$-module structure extends the $I$-module structure. It then follows immediately that $M$ is a simple $R$-module, as any $R$-submodule is an $I$-submodule.
\end{proof}

\begin{cor}\label{c:ideal.inherit}
Let $R$ be a ring and $I$ an ideal of $R$.  If $R$ is CCR (respectively, GCR), then so is $I$.
\end{cor}
\begin{proof}
Suppose first that $R$ is CCR and let $V$ be a simple $I$-module.  Put $D=\mathrm{End}_I(V)$.  By Lemma~\ref{l:restrict.ideals}, $V$ admits an $R$-module structure extending the $I$-module structure and $D=\mathrm{End}_R(V)$ by the first part of Lemma~\ref{l:restrict.ideals}.  Since $R$ is CCR, it follows that it acts on $V$ by finite rank operators over $D$ and hence the same is true for $I$.  We conclude that $I$ is CCR.

Next assume that $R$ is GCR and let $V$ be a simple $I$-module.  Set $D=\mathrm{End}_I(V)$. Again, by Lemma~\ref{l:restrict.ideals}, $V$ admits an $R$-module structure extending the $I$-module structure and $D=\mathrm{End}_R(V)$.  Since $R$ is GCR, there exists $r\in R$ acting as a non-zero finite rank operator (over $D$) on $V$.  Let $v\in V$ with $rv\neq 0$. Note that since $V$ is a simple $I$-module and $v\neq 0$, we must have that $v=sv$ with $s\in I$.  Then $(rs)v=r(sv)=rv\neq 0$ and $rs\in I$.  Since the elements acting on $V$ by finite rank operators form an ideal in $R$, we conclude that $rs\in I$ acts on $V$ as a non-zero finite rank operator over $D$.  Thus $I$ is GCR.  This completes the proof.
\end{proof}

\section{Morita invariance}\label{sec:Morita}
In this section, we improve Corollary~\ref{c:matrix.amp} by proving that, for rings with local units, being CCR or GCR is a Morita invariant property.  Recall that two rings $R$ and $S$ with local units are \emph{Morita equivalent} if their categories of unitary modules are equivalent. We shall use a result of~\cite{SS16}, which reduces the problem to checking corner invariance and invariance under matrix amplification (with finite index sets).  This section will not be used in the rest of the paper and can be omitted.

Our first step will be to develop some of Green's theory of an idempotent~\cite[Chapter~6]{Gr07} in the context of rings with local units.

\begin{lem}\label{l:goto.corner}
Let $R$ be a ring and $V$ a simple $R$-module.  Let $e\in R$ be an idempotent.  Then $eV=0$ or $eV$ is a simple $eRe$-module.  Moreover, if $eV\neq 0$, then the restriction map $\Phi\mapsto \Phi|_{eV}$ yields an isomorphism $\mathrm{End}_R(V)\to \mathrm{End}_{eRe}(eV)$.
\end{lem}
\begin{proof}
Suppose that $0\neq v\in eV$.  Then $Rv=V$ and so $eRev=eRv=eV$.  Thus $eV$ is a simple $eRe$-module.  Let $\Phi\in \mathrm{End}_R(V)$.  Then $\Phi(ev) = e\Phi(v)$ and so $\Phi(eV)\subseteq eV$.  Thus our restriction map  makes sense.  It is injective because $\mathrm{End}_R(V)$ is a division ring.  Let us check that it is surjective.

Fix $0\neq v\in eV$. First observe that if $\Psi\in \mathrm{End}_{eR}(eV)$, then $rv=0$ implies $r\Psi(v)=0$ for all $r\in R$.  Indeed, if $rv=0$, then for all $s\in R$, we have $esr\Psi(v) = esre\Psi(v)=\Psi(esrev)=\Psi(esrv)=\Psi(0)=0$.  Thus $eRr\Psi(v)=0$.    Suppose that $r\Psi(v)\neq 0$.  Then $Rr\Psi(v)=V$ by simplicity and so $0=eRr\Psi(v)=eV$, a contradiction.  It now follows that if we define $\Phi\colon V\to V$ by $\Phi(rv) = r\Psi(v)$ for $r\in R$, then $\Phi$ is well defined.  Indeed, if $rv=sv$, then $(r-s)v=0$ and so $(r-s)\Psi(v)=0$ and hence $r\Psi(v)=s\Psi(v)$.  Clearly, $\Phi$ is an $R$-module homomorphism.  Also if $w\in eV$, then $w=rv$ with $r\in eRe$ (by simplicity over $eRe$) and so $\Phi(w) = r\Psi(v) = \Psi(rv)=\Psi(w)$.  This concludes the proof of the lemma.
\end{proof}

Next we aim to show that if $R$ is a ring and $e$ is an idempotent of $R$, then every simple $eRe$-module is isomorphic to one of the form $eV$ with $V$ a simple $R$-module.

\begin{lem}\label{l:find.simple.help}
Let $R$ be a ring and $e$ an idempotent of $R$.  Suppose that $V$ is an $R$-module with $V=ReV$and $eV$ a simple $eRe$-module.
\begin{enumerate}
\item $N=\{v\in V\mid eRv=0\}$ is the unique maximal submodule of $V$.
\item $eN=0$.
\item $V/N$ is a simple $R$-module with $e(V/N)\cong eV$.
\end{enumerate}
\end{lem}
\begin{proof}
 Clearly, $N$ is an $R$-submodule and $eN=0$. Since $eV$ is simple, and hence non-zero, it follows that $N$ is a proper submodule.   Suppose that $W$ is a submodule not contained in $N$.  Then there exists $w\in W$ and $r\in R$ with $erw\neq 0$.  As $eV$ is a simple $eRe$-module, we must have that $eRerw=eV$ and so $V=ReV=ReRerw\subseteq W$.  Thus every proper submodule of $V$ is contained in $N$ and so $N$ is a maximal submodule.  Then $V/N$ is simple and
 $e(V/N)\cong eV/eN\cong eV$ as $eN=0$ and $M\mapsto eM$ is an exact functor.
 \end{proof}

\begin{cor}\label{c:get.a.simple}
Let $R$ be a ring and $e$ an idempotent of $R$.  Then every simple $eRe$-module is isomorphic to one of the form $eV$ with $V$ a simple $R$-module.
\end{cor}
\begin{proof}
Let $W$ be a simple $eRe$-module.
By Lemma~\ref{l:find.simple.help} it suffices to find an $R$-module $V$ with $ReV=V$ and $eV\cong W$.  Take $V=Re\otimes_{eRe} W$. Note that
since $re\otimes w= re(e\otimes w)$, it follows that $ReV=V$.  From \[e\left(\sum r_ie\otimes w_i\right) = \sum er_ie\otimes w_i = \sum e\otimes er_iew_i = e\otimes \sum er_iew_i\] we see that $eV$ consists of those elements of the form $e\otimes w$ with $w\in W$.  The mapping $\psi\colon W\to eV$ given by $\psi(w)= e\otimes w$ is an $eRe$-module homomorphism (as $rw\mapsto e\otimes rw = r\otimes w = r(e\otimes w)$ for $r\in eRe$).  To see that it is an isomorphism, define a mapping $Re\times W\to W$ by $(r,w)\mapsto erw$ for $r\in Re$.  This map is $eRe$-balanced since if $s\in eRe$, then $(rs,w)\mapsto ersw$ and $(r,sw)\mapsto ersw$.  It is clearly bilinear and so it induces a well-defined mapping $Re\otimes_{eRe} W\to W$ sending $r\otimes w$ to $erw$.  In particular, $e\otimes w$ maps to $ew=w$ and so the restriction to $eV$ provides an inverse to $\psi$.
\end{proof}

A property of rings is said to be \emph{corner invariant}~\cite{SS16} if $R$ has the property if and only if all its corners $eRe$ with $e$ an idempotent have the property.  We are  now prepared to prove that CCR and GCR are corner invariant for rings with local units.

\begin{prop}\label{p:corner.inv}
Let $R$ be a ring with local units.  Then $R$ is CCR (respectively, GCR) if and only if each corner $eRe$ is CCR (respectively, GCR).
\end{prop}
\begin{proof}
Let  $W$ be a simple $eRe$-module.  Then by Corollary~\ref{c:get.a.simple} we have that $W\cong eV$ with $V$ a simple $R$-module.  Moreover, if $D=\mathrm{End}_R(V)$, then $\Phi\mapsto \Phi|_{eV}$ gives an isomorphism from $D$ to $\mathrm{End}_{eRe}(eV)$ by Lemma~\ref{l:goto.corner}.  If $r\in eRe$, then $rV = reV$ and so $r$ has finite rank over $D$ as an operator on $V$ if and only if it has finite rank as an operator on $eV$ over $D$.  It now follows easily that if $R$ is CCR, then so is $eRe$ by Proposition~\ref{p:ccrgcrnice}.  Suppose that $R$ is GCR. Note that since $eV\neq 0$, we must have $ReV=V$ by simplicity of $V$.  Let $r\in R$ act as a non-zero finite rank operator over $D$.  Then there exists $v\in V$ and $s\in R$ such that $rsev\neq 0$ using that $ReV=V$.  As $Rrsev=V$, we can find $t\in R$ with $0\neq trsev\in eV$.  Then $etrsev\neq 0$ and $etrse\in eRe$ is a finite rank operator on $V$ and hence $eV$.  We conclude that $eRe$ is GCR by Proposition~\ref{p:ccrgcrnice}.

Next suppose that each corner of $R$ is CCR and let $V$ be a simple $R$-module.  Let $r\in R$ with $rV\neq 0$.  Then there is an idempotent $e$ with $r\in eRe$ and so $0\neq rV\subseteq eV$.  Thus $eV$ is a simple $eRe$-module and the restriction map $\Phi\mapsto \Phi|_{eV}$ is an isomorphism of $D=\mathrm{End}_R(V)$ with $\mathrm{End}_{eRe}(eV)$ by Lemma~\ref{l:goto.corner}.  It follows that $r$ acts on $eV$ as a non-zero finite rank operator over $D$.  But $rV=reV$ and so $r$ acts on $V$ as a non-zero finite rank operator over $D$.  We conclude that $R$ is CCR.

Finally, suppose that each corner of $R$ is GCR and let $V$ be a simple $R$-module. Then since $V$ is a unitary $R$-module, we must have $eV\neq 0$ for some idempotent $e\in R$.  Then $eV$ is a simple $eRe$-module and the restriction map $\Phi\mapsto \Phi|_{eV}$ is an isomorphism of $D=\mathrm{End}_R(V)$ with $\mathrm{End}_{eRe}(eV)$ by Lemma~\ref{l:goto.corner}.  Some element $r\in eRe$ acts on $eV$ as a non-zero finite rank operator over $D$.  Hence, since $rV=reV$, we have that $r$ acts a non-zero finite rank operator on $V$ over $D$.  Thus $R$ is GCR.
\end{proof}

According to the main result of~\cite{SS16}, to show that CCR and GCR are Morita invariant properties for rings with local units it suffices to show that they are corner invariant, which was done in Proposition~\ref{p:corner.inv}, and that $R$ has the property if and only if $M_n(R)$ has the property for every $n\geq 1$.

\begin{prop}\label{p:matrix.inv}
Let $R$ be a ring with local units.  Then $R$ is CCR  (respectively, GCR) if and only if $M_n(R)$ is CCR (respectively, GCR) for any $n\geq 1$.
\end{prop}
\begin{proof}
Let $n\geq 1$.  Then $M_n(R)$ has local units because if $E$ is a set of local units for $R$, then the set of diagonal matrices $\mathrm{diag}(e,e,\ldots, e)$ with $e\in E$ is a set of local units for $M_n(R)$.

Suppose first that $M_n(R)$ is CCR (respectively, GCR).  Then each corner in $M_n(R)$ is CCR (respectively, GCR) by Proposition~\ref{p:corner.inv}.  But if $e$ is an idempotent of $R$, then $eRe\cong PM_n(R)P$ where $P$ is the idempotent matrix with $e$ in the upper left corner and $0$ in all other entries.  Thus each corner of $R$ is CCR (respectively, GCR) and hence $R$ is CCR (respectively, GCR) by Proposition~\ref{p:corner.inv}.

Next suppose that $R$ is CCR (respectively, GCR).  Let $P=P^2$ be an idempotent in $M_n(R)$.  Then there is $e\in E$ such that all the entries of $P$ belong to $eRe$.   Then $P$ is an idempotent of $M_n(eRe)$ and $PM_n(R)P=PM_n(eRe)P$.  Since $eRe$ is CCR (respectively, GCR) by Proposition~\ref{p:corner.inv}, we deduce that $M_n(eRe)$ is CCR (respectively, GCR) by Corollary~\ref{c:matrix.amp} and hence $PM_n(R)P=PM_n(eRe)P$ is CCR (respectively, GCR) by Proposition~\ref{p:corner.inv}.  This completes the proof by another application of Proposition~\ref{p:corner.inv}.
\end{proof}

Applying~\cite[Theorem~5.4]{SS16}, we may draw the following conclusion.

\begin{thm}\label{t:morita}
Let $R$ and $S$ be Morita equivalent rings with local units.  Then $R$ is CCR (respectively, GCR) if and only if $S$ is CCR (respectively, GCR).
\end{thm}

\begin{remark}\label{r:morita.ref}
The argument in Proposition~\ref{p:matrix.inv} can easily be abstracted to give the following conclusion. Suppose
that $\mathcal P$ is a property such that a ring $R$
with local units satisfies $\mathcal P$ if and only if all its corners $eRe$ satisfy $\mathcal P$. Then
$\mathcal P$ is a Morita-invariant property for rings with local units if and only if
whenever $R$ is a unital ring with property $\mathcal P$ then $M_n(R)$  also has $\mathcal P$ for all $n\geq 1$.
\end{remark}

\section{CCR group rings} \label{sec:CCRgrouprings}
Recall that a group $G$ is \emph{virtually abelian} if it contains an abelian subgroup of finite index (which we may assume is normal without loss of generality).
It is an observation, going back essentially to Kaplansky~\cite{K49}, that if $H\lhd G$ is an abelian normal subgroup of index $n$ and $\Bbbk$ is an algebraically closed field of cardinality greater than $|G|$, then each simple $\Bbbk G$-module is of dimension at most $n$ over $\Bbbk$.  From this, it immediately follows that $\Bbbk G$ is CCR.  We shall need the following version of Clifford's theorem for the proof.

\begin{thm}[Clifford]\label{t:clifford}
Let $G$ be a group, $\Bbbk$ a field and $H\lhd G$ a normal subgroup of index $n$.    If $V$ is a simple $\Bbbk G$-module, then as a $\Bbbk H$-module $V$ embeds in a direct sum  of at most $n$ simple $\Bbbk H$-modules and, in particular, is semisimple.
\end{thm}
\begin{proof}
Let $T$ be a set of coset representatives of $H$ in $G$ with $1\in T$.  First note that $V$ is finitely generated as a $\Bbbk H$-module.  Indeed, if $v\neq 0$, then $\Bbbk Gv=V$ by simplicity.  But then using the decomposition into cosets, we see that $V= \sum_{t\in T}\Bbbk Htv$ and so $V$ is generated by $\{tv\mid t\in T\}$.  It follows from Zorn's lemma that $V$ has a simple quotient $\Bbbk H$-module $V/K$ with $K\leq V$ a $\Bbbk H$-submodule.  If $t\in T$, then  $tK$ is a $\Bbbk H$-submodule since if $h\in H$ and $v\in K$, then $htv = t(t^{-1}ht)v\in tK$ by normality of $H$.    If $g\in G$ and $t\in T$, write $gt=t'h$ with $t'\in T$ and $h\in H$.  Then $gtK = t'hK=t'K$. It follows that $M=\bigcap_{t\in T}tK\leq K<V$ is a proper $\Bbbk G$-submodule and hence $0$ by simplicity.  Thus we have an embedding
\[V\hookrightarrow \bigoplus_{t\in T} V/tK.\]  Each $V/tK$ is a simple $\mathbb KH$-module, for if $tK\leq W\leq V$ is a $\Bbbk H$-submodule, then $K\leq t^{-1}W\leq V$ is a $\Bbbk H$-submodule and hence by simplicity of $V/K$, we have $K=t^{-1}W$ or $V=t^{-1}W$ and so $W=tK$ or $W=V$.   Thus $V$ is a submodule of a direct sum of $n$ simple $\Bbbk H$-modules and hence is a semisimple.
\end{proof}

\begin{cor}\label{c:CCR.group}
Let $\Bbbk$ be an algebraically closed field and $G$ a virtually abelian group of cardinality less than that of $\Bbbk$.  Suppose that $G$ has a normal abelian subgroup $H$ of index $n$.  Then each simple $\Bbbk G$-module has dimension at most $n$ over $\Bbbk$ and so $\Bbbk G$ is CCR.
\end{cor}
\begin{proof}
First observe that if $A$ is an abelian group of cardinality less than that of $\Bbbk$, then every simple $\Bbbk A$-module is one-dimensional.  Indeed, if $V$ is a simple $\Bbbk A$-module, then $\mathrm{End}_{\Bbbk A}(V)=\Bbbk$ by Lemma~\ref{l:schur.dim}.  But since $A$ is abelian, the action of $A$ on $V$ is contained in $\mathrm{End}_{\Bbbk A}(V)=\Bbbk$  and hence $A$ acts by scalar multiplications.  But then $V$ must be one-dimensional by simplicity as any subspace is $A$-invariant.

For the general case, we apply the previous observation to conclude that every simple $\Bbbk H$-module is one-dimensional and then apply Theorem~\ref{t:clifford} to deduce that the dimension over $\Bbbk$ of any simple $\Bbbk G$-module is at most $n$.  It follows that $\Bbbk G$ is CCR by Proposition~\ref{p:algclosecase}.
\end{proof}

\begin{cor}\label{c:CG.complex}
Let $G$ be a countable virtually abelian group.  Then $\mathbb CG$ is CCR.
\end{cor}

We remark that it follows from Burnside's theorem, or the density theorem, that if $\rho$ is a finite dimensional irreducible representation of an algebra over an algebraically closed field, then every matrix is in the image of $\rho$.

Let $G$ be a countable group.  One would expect that all irreducible representations of $G$ over $\mathbb C$ are finite dimensional, i.e., $\mathbb CG$ is CCR, if and only if $G$ is virtually abelian.    First note that the group $C^*$-algebra of $G$ is CCR if and only if $G$ is virtually abelian~\cite{Th64}.  Corollary~\ref{c:CCR.group} shows that if $G$ is virtually abelian, then the dimensions of its irreducible complex representations are bounded.  Isaacs and Passman~\cite{IP64} show that this property characterizes virtually abelian groups. Moreover, it is known that a countable group that is solvable, linear or torsion has only finite dimensional complex irreducible representations if and only if it is virtually abelian~\cite{PT00,Sn06}.

In the case of positive characteristic, there are other examples of CCR group rings.  Recall that a group $G$ is \emph{polycyclic} if it has a subnormal series
\[\{1\}=G_0\leq G_1\leq \cdots \leq G_n=G\] with $G_i/G_{i-1}$ cyclic for $i\geq 1$.  A polycyclic group is necessarily finitely generated and solvable.  All finitely generated nilpotent groups are polycyclic.  A group is \emph{virtually polycyclic} if it has a finite index polycyclic subgroup.  The following theorem was proved by Roseblade~\cite{roseblade}.

\begin{thm} \label{t:Rosebade}
Let $\Bbbk$ be an algebraic extension of a finite field and $G$ a virtually polycyclic group.  Then every simple $\Bbbk G$-module is finite dimensional.  Consequently, $\Bbbk G$ is CCR.
\end{thm}

\section{Groupoids and Steinberg algebras} \label{sec:steinbergalgs}
Here we review some basic facts about ample groupoids and their algebras.  See~\cite{CFST14,St10} for details.

\subsection{Ample groupoids}
An \emph{\'etale} groupoid is a topological groupoid $\mathscr G$ such that source map $s\colon \mathscr G^1\to \mathscr G^0$ is a local homeomorphism, whence also the range map $r$ and the multiplication map are local homeomorphisms.  If $\mathscr G^0$ is a Hausdorff space with a basis of compact open sets, then $\mathscr G$ is called \emph{ample} following Paterson~\cite{Pa99}.  We shall always use the term compact in this paper to mean a Hausdorff space where open covers have finite subcovers, as in Bourbaki.  However, we do not require locally compact spaces to be Hausdorff.  In particular, if $\mathscr G$ is an ample groupoid, then $\mathscr G^1$ will be locally compact and locally Hausdorff but need not be Hausdorff.  We view the space of objects $\mathscr G^0$ as the subspace of $\mathscr G^1$ consisting of units, that is, as the unit space.  It is an open subspace when $\mathscr G$ is \'etale.

An open subset $U\subseteq \mathscr G^1$ is called a \emph{bisection} if $s$ and $r$ are injective when restricted to $U$.  The compact open bisections of an ample groupoid $\mathscr G$ form a basis for the topology.  The set $S$ of compact open bisections is an inverse semigroup with multiplication and inversion defined by
\begin{align*}
& U V= \{\gamma_1\gamma_2\mid \gamma_1\in U, \gamma_2\in V\}, \text{ and} \\
& U^{-1}=\{\gamma^{-1}\mid \gamma\in U\},
\end{align*}
respectively; see~\cite{Pa99}. The idempotents in $S$ are given by compact open subsets of the unit space $\G^0$. See~\cite{Lawson} for more on inverse semigroups.

Note that $S$ acts on $\G^0$ by partial homeomorphisms: take as the domain of $U$ the set $s(U)=U^{-1}U$ and as the range the set $r(U)=UU^{-1}$. For the action consider any $u \in s(U)$. Then since $U$ is a bisection there exists a unique $\gamma \in U$ such that $\gamma\colon u\to v$ (that is, $s(\gamma)=u$ and  $r(\gamma)=v$ in $\G$). Define the action  by \[ U\cdot u =v.\]

If $\mathscr G$ is a groupoid, we define the orbit equivalence relation on $\mathscr G^0$ by $u\sim v$ if there exists $\gamma\colon u\to v$.  The quotient space, called the \emph{orbit space}, is denoted $\mathscr G^0/\mathscr G$ and equivalence classes are called \emph{orbits}.

The orbit $\orb_u$ of $u$ in $\G^0$ is the same as the orbit $S\cdot u$ of $u$ for the action of the inverse semigroup $S$ in the case that $\mathscr G$ is ample.  That $S\cdot u\subseteq \orb_u$ is clear from the definition.  If $\gamma\colon u\to v$, then there is a compact open bisection $U$ containing $\gamma$ and $Uu=v$.  Thus $\orb_u=S\cdot u$.

A subset $X\subseteq \mathscr G^0$ is said to be \emph{invariant} if it is a union of orbits; this is equivalent to being invariant under the action of  the inverse semigroup $S$.
If $X\subseteq \mathscr G^0$ is invariant, then the \emph{reduction} $\mathscr G|_X$ is the subgroupoid with object set $X$ and arrows those belonging to $s^{-1}(X)$; we endow it with the relative topology.

The \emph{isotropy group} $G_u$ at a unit $u\in \mathscr G^0$ is the group of all arrows $\gamma\colon u\to u$.  It is discrete in the relative topology for an \'etale groupoid.  If $u\sim v$, then $G_u\cong G_v$.  That is, the isotropy group depends only on the orbit (up to isomorphism).

A groupoid is said to be \emph{transitive} if it has a single orbit, i.e., its orbit space is a point.

\begin{lem}\label{lem:discreteorbits}
	Let $\G$ be a second-countable ample groupoid (meaning $\mathscr G^1$ is second countable).
\begin{enumerate}	
\item The orbits of $\G$ are countable.
\item An orbit is discrete  if and only if it contains an isolated point (in the relative topology).
\item Any orbit of $\G$ which is a Baire space in the relative topology is discrete (in the relative topology).  This applies, in particular, to open orbits and closed orbits.
\end{enumerate}
\end{lem}
	\begin{proof}
		Since $\G$ is second-countable, it follows that we have countably many compact open bisections. Let $S$ denote the inverse semigroup of compact open bisections in $\G$ acting on $\G^0$ by partial homeomorphisms; note that $S$ is countable.  We let $\Phi_U\colon s(U)\to r(U)$ be the partial homeomorphism of $\mathscr G^0$ corresponding to $U\in S$.  Fix $u\in\mathscr G^0$. Since $\orb_u=S\cdot u$ and $S$ is countable, it follows that $\orb_u$ is also countable.
		
Next we show that $\orb_u$ contains an isolated point if and only if it is discrete (in the relative topology). To see this suppose that $v\in\orb_u$ is an isolated point in the relative topology, that is,  $W\cap \orb_u=\{v\}$ for some $W\subseteq \mathscr G^0$ open. Then, for any $w\in\orb_u$, there is a compact open bisection $U\in S$ with $s(U)$ containing $v$ such that $U\cdot v=w$, i.e., $\Phi_U(v)=w$. Then $\Phi_U(W\cap s(U))$ is open and $\Phi_U(W\cap s(U))\cap \orb_u = \{w\}$.  We conclude that $\orb_u$ is discrete in the relative topology.
		
We show that if $\orb_u$ is a Baire space in the relative topology, then the orbit is discrete.  Suppose that the orbit is not discrete.   Since $\orb_u$ is Hausdorff, singletons are closed, and since $\orb_u$ is not discrete, singletons cannot be isolated points by the above paragraph. That is, singletons are  closed nowhere dense sets. Then
		\[\orb_u=\bigcup_{v\in\orb_u}\{v\}\] is a countable union of closed nowhere dense sets,
		which contradicts that $\orb_u$ is a Baire space.   Hence $\orb_u$ is discrete.   If $\orb_u$ is either open or closed, then $\orb_u$ is second-countable, locally compact and Hausdorff in the relative topology and hence a Baire space.
\end{proof}

Note that the first item and the second part of the third item of Lemma~\ref{lem:discreteorbits} are false without the hypothesis of second countability.  Let $C_2$ be the two-element cyclic group and $G=C_2^{\mathbb N}$ be a countable direct product of copies of $C_2$.   Let $X=C_2^{\mathbb N}$ with the product topology of discrete spaces; so $X$ is a Cantor set.  Then $G$ acts on $X$ by the regular action, which is by homeomorphisms; however, we view $G$ with the discrete topology.  Then the transformation groupoid $\mathscr G = G\ltimes X$ is ample, not second-countable and consists of a single closed orbit which is uncountable and not discrete (in fact, compact).

If $\mathscr G$ is second-countable, then all the isotropy groups of $\mathscr G$  are countable, since they are discrete and second-countable in the relative topology.

We shall also need the following version of the Mackey-Glimm dichotomy; see~\cite{Ram90}.  %A proof is provided in the appendix.
We provide a short proof for the special case of ample groupoids, in part because standard references assume  that the groupoid is Hausdorff. Recall that a subspace $X$ of a space $Y$ is \emph{locally closed} if it is open in its closure in the relative topology.

\begin{prop}\label{p:locally.closed}
If $\mathscr G$ is a second-countable ample groupoid.  Then the following are equivalent.
\begin{enumerate}
\item  $\mathscr G^0/\mathscr G$ is $T_0$.
\item Each orbit of $\G$ is a Baire space in the relative topology.
\item Each orbit of $\G$ is discrete in the relative topology.
\item Each orbit is locally closed.
\end{enumerate}
\end{prop}
\begin{proof}
Note that if $q\colon \G^0\to \G^0/\G$ is the quotient map, then $q$ is open because $q^{-1}(q(U)) = r(s^{-1}(U))$ is open.
We show that (1) implies (2).  Assume $\G^0/\G$ is $T_0$.  Let $\{U_n\mid n\in \mathbb N\}$ be a countable basis for the topology on $\G^1$.  Let $u\in\G^0$ and note that $\overline{\orb_u}$ is a Baire space, being a closed subspace of a second-countable locally compact Hausdorff space.  Replacing $\G$ by $\G|_{\overline{\orb_u}}$, we may assume without loss of generality that $\orb_u$ is dense in $\G^0$.  Put $V_n =r(s^{-1}(U_n))$ and note that the $V_n$ are open invariant sets containing $\orb_u$ by density.  Suppose $v\notin \orb_u$.  Then since $\orb_v\subseteq \overline{\orb_u}$, we conclude that $q(v)\in \overline{q(u)}$.  Since $\G^0/\G$ is $T_0$ and $q(u)\neq q(v)$, we must have an open subset $U$ of $\G^0/\G$ with $q(u)\in U$ and $q(v)\notin U$.  Then $q^{-1}(U)$ is an open invariant subset of $\G^0$ containing $\orb_u$ but not intersecting $\orb_v$.  If $u\in U_n\subseteq q^{-1}(U)$, then $V_n\subseteq q^{-1}(U)$ and so disjoint from $\orb_v$.  It follows that $\bigcap_{n\in \mathbb N} V_n = \orb_u$.  Therefore, $\orb_u$ is a $G_{\delta}$-set in a second-countable locally compact Hausdorff space and hence a Baire space in the relative topology by~\cite[Lemma~6.4]{CrossedProd}.

(2) implies (3) by Lemma~\ref{lem:discreteorbits}.  Assume that (3) holds and let $u\in \G^0$.  Since $\orb_u$ is discrete in the relative topology, there is a neighborhood $U$ with $U\cap \orb_u = \{u\}$.  Suppose that $v\in U$ with $v\neq u$.  Since $\G^0$ is Hausdorff, there is a neighborhood $V$ of $v$ with $V\subseteq U$ and $u\notin V$.  But then $V$ is a neighborhood of $v$ failing to intersect $\orb_u$ and so $v\notin \overline{\orb_u}$.  Therefore, $U\cap \overline{\orb_u} = \{u\}$ and so $W=r(s^{-1}(U))$ is an open invariant set with $W\cap \overline{\orb_u}=\orb_u$.  Thus $\orb_u$ is locally closed.

To see that (4)  implies (1), note that if $\overline{q(u)}=\overline{q(v)}$, then $\overline{\orb_u}=\overline{\orb_v}$.  Since $\orb_u$ is open in $\overline{\orb_u}$ in the relative topology, there is an open subset $U$ with $\emptyset\neq U\cap \overline{\orb_u} \subseteq \orb_u$.  Since $\orb_u\subseteq \overline{\orb_v}$ we must have $\emptyset\neq U\cap \orb_v\subseteq U\cap \overline{\orb_v} = U\cap \overline{\orb_u} \subseteq \orb_u$.  Thus $\orb_u=\orb_v$ and so $\G^0/\G$ is $T_0$.
\end{proof}

\subsection{Steinberg algebras}
Let $\Bbbk$ be a field and $\mathscr G$ an ample groupoid.  The \emph{Steinberg algebra}\footnote{The first and third authors insist on using this name.} $\Bbbk \mathscr G$ is the $\Bbbk$-span of the characteristic functions $1_U$ of compact open bisections $U$ with the convolution product
\[f\ast g(\gamma) =\sum_{s(\alpha)=s(\gamma)}f(\gamma\alpha^{-1})g(\alpha)\] (which is easily checked to be a finite sum).  If $G$ is a discrete group, viewed as a one-object ample groupoid, then $\Bbbk G$ is the usual group algebra.

Note that $1_U\ast 1_V=1_{UV}$ for compact open bisections and so $\Bbbk \mathscr G$ is a quotient of the semigroup algebra $\Bbbk S$, where $S$ is the inverse semigroup of compact open bisections.  Therefore, the dimension of $\Bbbk\mathscr G$ is bounded above by the cardinality of $S$.  In particular, if $\mathscr G$ is second-countable, then $\Bbbk \mathscr G$ is of countable dimension over $\Bbbk$.

Every Steinberg algebra has an involution $f\mapsto f^*$ given by $f^*(\gamma) = f(\gamma^{-1})$ and satisfying $(f\ast g)^{\ast}= g^{\ast}\ast f^{\ast}$.  It follows that there is no difference between left and right CCR or GCR for Steinberg algebras.

If $\mathscr G^1$ is Hausdorff, then $\Bbbk \mathscr G$ consists precisely of the locally constant mappings $f\colon \mathscr G^1\to \Bbbk$ with compact support.  In particular, if $\mathscr G$ is discrete it consists of the finitely supported functions. In general, if $U\subseteq \mathscr G^1$ is compact open, then $1_U\in \Bbbk\mathscr G$ regardless of whether $\mathscr G^1$ is Hausdorff.

The following result is folklore.

\begin{prop}\label{p:discrete.groupoid}
Let $\Bbbk$ be a field and $\mathscr G$ a discrete transitive groupoid with isotropy group $G_u$ at $u\in \mathscr G^0$.  Then $\Bbbk \mathscr G\cong M_{\mathscr G^0}(\Bbbk G_u)$.  Hence $\Bbbk \mathscr G$ is CCR (respectively, GCR) if and only if $\Bbbk G_u$ is CCR (respectively, GCR).
\end{prop}
\begin{proof}
Fix, for each $v\in \mathscr G^0$, an arrow $\gamma_v\colon u\to v$.  We may assume without loss of generality that $\gamma_u=u$ is a unit.  Since $\mathscr G$ is discrete, it follows that $\Bbbk\mathscr G$ consists of the finitely supported functions and so we can identify $\Bbbk \mathscr G$ with the $\Bbbk$-vector space with basis $\mathscr G^1$ and with the unique product extending the product in $\mathscr G$ (where undefined products are declared to be $0$) and satisfying the distributive law.

If $v,w\in \mathscr G^0$, let $E_{ vw}$ be the matrix with $1$ in position $vw$ and $0$ in all other positions.  Note that the $\gamma E_{vw}$ with $\gamma\in G_u$ form a $\Bbbk$-basis for $M_{\mathscr G^0}(\Bbbk G_u)$.  Define $\Phi\colon \Bbbk\mathscr G\to M_{\mathscr G^0}(\Bbbk G_u)$ by \[\Phi(\gamma) = \gamma_v^{-1}\gamma\gamma_wE_{vw}\] for $\gamma\colon w\to v$.  Then it is easy to verify that $\Phi$ is an isomorphism with inverse given by $\Phi^{-1}(\gamma E_{vw}) = \gamma_v\gamma\gamma_w^{-1}$.

The final statement follows from Corollary~\ref{c:matrix.amp}.
\end{proof}

A Steinberg algebra $\Bbbk\mathscr G$ is unital if and only if the unit space $\mathscr G^0$ is compact.  It is always a ring with local units where a set of  local units is given by the $1_U$ with $U\subseteq \mathscr G^0$ compact open.

If $U$ is an open invariant subset, then $\mathscr G|_U$ is an open subgroupoid of $\mathscr G$ and hence ample in its own right.  We can identify $\Bbbk \mathscr G|_U$ as a subalgebra of $\Bbbk \mathscr G$, and in fact it is easy to see that it is an ideal.

 If $X$ is a closed invariant subset, then $\mathscr G|_X$ is a closed subgroupoid and is an ample groupoid in the relative topology.  Moreover, for any $f\in \Bbbk \mathscr G$, its restriction to $\mathscr G|_X$ belongs to $\Bbbk \mathscr G|_X$ and the restriction map is a homomorphism with kernel the ideal $\Bbbk \mathscr G|_{\mathscr G^0\setminus X}\subseteq \Bbbk \mathscr G$.  Indeed, if $U$ is a compact open bisection of $\mathscr G$, then $U\cap \mathscr G|_X$ is open in the induced topology and is also a closed subspace of the compact Hausdorff space $U$ and hence compact. Thus $U\cap \mathscr G|_X$ is a compact open bisection of $\mathscr G|_X$.  The restriction homomorphism is, in fact, surjective. The argument below is adapted from~\cite[Lemma~5.5]{lawsonvdovina}. Let $U$ be a compact open bisection of $\mathscr G|_X$. Then we can cover $U$ by a family of compact open bisections of $\mathscr G$ whose intersections with $\mathscr G|_X$ are contained in $U$.  Then by going to a finite subcover, we see that $U = (U_1\cup \cdots \cup U_n)\cap \mathscr G|_X$ for some compact open bisections $U_1,\ldots, U_n$ of $\mathscr G$. Put $V_1=U_1$ and suppose inductively that we have found a compact open bisection $V_i$ of $\mathscr G$ with $V_i\cap \mathscr G|_X=(U_1\cup\cdots \cup U_i)\cap \mathscr G|_X$ with $1\leq i\leq n-1$.  Note that $W=s(V_i)\setminus s(U_{i+1})$ and $W' = r(V_i)\setminus r(U_{i+1})$ are compact open subsets of $\mathscr G^0$.  Then $V'=W'V_iW\subseteq V_i$ is a compact open bisection and $s(V')\cap s(U_{i+1})=\emptyset =r(V')\cap r(U_{i+1})$ and so $V_{i+1} =V'\cup U_{i+1}$ is a compact open bisection of $\mathscr G$.  We claim that $V_{i+1}\cap \mathscr G|_X = (U_1\cup\cdots\cup  U_{i+1})\cap \mathscr G|_X$.  The containment from left to right is obvious since $V'\subseteq V_i$.    Also $U_{i+1}\cap \mathscr G|_X\subseteq V_{i+1}\cap \mathscr G|_X$ by construction.  Suppose that $\gamma\in (U_1\cup\cdots \cup U_i)\cap \mathscr G|_X=V_i\cap \mathscr G|_X\subseteq U$ but $\gamma\notin U_{i+1}$.  If $s(\gamma)\in s(U_{i+1})$, then since $U$ is a bisection of $\mathscr G|_X$, we deduce that $\gamma\in U_{i+1}$.  Thus $s(\gamma)\in s(V_i)\setminus s(U_{i+1})$.  An identical argument shows that $r(\gamma)\in r(V_i)\setminus r(U_{i+1})$.  Therefore, $\gamma\in V'\subseteq V_{i+1}$.  Thus $V_n$ is a compact open bisection of $\mathscr G$ with $V_n\cap \mathscr G|_X = U$ and so  $1_U$ is the restriction of $1_{V_n}$. We conclude that the restriction map is onto.

If $u\in \mathscr G^0$, then $\mathscr G_u$ denotes the set of arrows with source $u$.  The isotropy group $G_u$ acts freely on the right of $\mathscr G_u$ by multiplication.  Thus $\Bbbk \mathscr G_u$ is  a free right $\Bbbk G_u$-module.  In fact, it is a $\Bbbk \mathscr G$-$\Bbbk G_u$-bimodule where the left $\Bbbk \mathscr G$-module structure is defined by
\[f\gamma = \sum_{s(\alpha)=r(\gamma)} f(\alpha)\alpha\gamma\] for $f\in \Bbbk\mathscr G$ and $\gamma\in \mathscr G_u$.  We then have an induction functor
\[\mathrm{Ind}_u\colon \Bbbk G_u\text{-}\mathrm{mod}\to \Bbbk \mathscr G\text{-}\mathrm{mod}\] given by
\[\mathrm{Ind}_u(V) = \Bbbk \mathscr G_u\otimes_{\Bbbk G_u} V.\]  Since $\Bbbk \mathscr G_u$ is a free right $\Bbbk G_u$-module, it follows that $\mathrm{Ind}_u$ is an exact functor.  Moreover, a $\Bbbk G_u$-basis for $\Bbbk\mathscr G_u$ is given by fixing for each $v\in \orb_u$ an arrow $\gamma_v\colon u\to v$.  It follows that as a $\Bbbk$-vector space, we have that
\begin{equation}\label{eq:induced.decomp}
\mathrm{Ind}_u(V) = \bigoplus_{v\in \orb_u} \gamma_v\otimes V.
\end{equation}

A key fact, proved in~\cite{St10}, is that $\mathrm{Ind}_u$ preserves simplicity.

\begin{thm}
If $V$ is a simple $\Bbbk G_u$-module, then $\mathrm{Ind}_u(V)$ is a simple $\Bbbk\mathscr G$-module.
\end{thm}

Of particular interest is when $V$ is the trivial module $\Bbbk$ (where $G_u$ acts as the identity).  Then \eqref{eq:induced.decomp} implies that as a vector space $\mathrm{Ind}_u(\Bbbk)\cong \Bbbk \orb_u$.  The action is easily verified to be given by
\[fv = \sum_{\gamma\colon v\to w}f(\gamma)w\] for $f\in \Bbbk \mathscr G$ and $v\in \orb_u$ since $\alpha(\gamma_{s(\alpha)}\otimes 1) = \gamma_{r(\alpha)}\otimes 1$.  In particular, if $U$ is a compact open bisection, then
\[1_Uv = \begin{cases} U\cdot v, & \text{if}\ v\in s(U)\\ 0, & \text{else.} \end{cases}\]
One can easily prove the simplicity of $\Bbbk \orb_u$ directly from these observations.  If $x=\sum_{v\in \orb_u}c_vv$ is a non-zero finite linear combination, then since $\mathscr G^0$ is Hausdorff, we can find $U\subseteq \mathscr G^0$ compact open with $Ux=c_vv$ with $c_v\neq 0$ for some $v\in \orb_u$.  Then using that $S$ acts transitively on $\orb_u$, we can obtain any basis element from $c_vv$.

 It is not difficult to verify that $\mathrm{End}_{\Bbbk G_u}(V,V)\cong \mathrm{End}_{\Bbbk \mathscr G}(\mathrm{Ind}_u(V),\mathrm{Ind}_u(V))$ for any $\Bbbk G_u$-module $V$.  We shall only perform this verification for the case of the trivial module, as that is all we shall need.

\begin{lem}\label{l:endo.trivial}
Let $\Bbbk$ be a field and $\orb_u$ an orbit of $\mathscr G$.  Then $\mathrm{End}_{\Bbbk \mathscr G}(\Bbbk \orb_u)=\Bbbk$.
\end{lem}
\begin{proof}
Let $\Phi\in \mathrm{End}_{\Bbbk \mathscr G}(\Bbbk \orb_u)$.  We first claim that $\Phi(u) = cu$ for some $c\in \Bbbk$.  Indeed, if $\Phi(u) = \sum_{v\in \orb_u}c_vv$, then since the support of the sum is finite and $\mathscr G^0$ is Hausdorff, we can find a compact open neighborhood $U$ of $u$ in $\mathscr G^0$ so that if $v\in U\cap \orb_u$ and $c_v\neq 0$, then $v=u$.  As $1_Uu=u$, it then follows that $\Phi(u) = \Phi(1_Uu)=1_U\Phi(u) = c_uu$.    Next observe that if $v\in \orb_u$, then there is a compact open bisection $V$ with $V\cdot u=v$, i.e., $1_Vu=v$.  Then $\Phi(v) = \Phi(1_Vu)=1_V\Phi(u) = 1_Vc_uu = c_u1_Vu=c_uv$.    Thus $\Phi$ is multiplication by $c_u\in \Bbbk$, as required.
\end{proof}

\subsection{Modules and sheaves}
 In~\cite{St14} the second author showed that the category $\Bbbk \mathscr G\text{-}\mathrm{mod}$ of unitary $\Bbbk\mathscr G$-modules is equivalent to the category of $\mathscr G$-sheaves of $\Bbbk$-vector spaces.  We shall need aspects of this equivalence.

A \emph{$\mathscr G$-sheaf} consists of a space $E$, a local homeomorphism $p\colon E\to \mathscr G^0$ and an action map $\mathscr G^1\times_{\mathscr G^0} E\to E$ (where the fiber product is with respect to $s$ and $p$), written $(\gamma,e)\mapsto \gamma e$ satisfying the following axioms:
\begin{itemize}
\item $p(e)e=e$ for all $e\in E$;
\item $p(\gamma e)=r(\gamma)$ whenever $p(e)=s(\gamma)$;
\item $\gamma(\alpha e)=(\gamma \alpha)e$ whenever $p(e)=s(\alpha)$ and $s(\gamma)=r(\alpha)$.
\end{itemize}

A \emph{$\mathscr G$-sheaf of $\Bbbk$-vector spaces} is a $\mathscr G$-sheaf $(E,p)$ together with a $\Bbbk$-vector space structure on each stalk $E_u=p^{-1}(u)$ such that:
\begin{itemize}
\item the zero section $0\colon \mathscr G^0\to E$ given by $u\mapsto 0_u$ (the zero of $E_u$) is continuous;
\item (fiberwise) addition $E\times_{\mathscr G^0} E\to E$ is continuous;
\item scalar multiplication $\Bbbk \times E\to E$ is continuous;
\item for each $\gamma\in \mathscr G^1$, the map $L_{\gamma}\colon E_{s(\gamma)}\to E_{r(\gamma)}$ given by $L_{\gamma}(e) = \gamma e$ is linear;
\end{itemize}
where $\Bbbk$ has the discrete topology in the third item.
Note that the first three conditions are equivalent to $(E,p)$ being a sheaf of $\Bbbk$-vector spaces over $\mathscr G^0$.  Also observe that each $L_{\gamma}$ is a linear isomorphism with inverse $L_{\gamma^{-1}}$ and that $L_u = 1_{E_u}$ for a unit $u\in \mathscr G^0$.  There is an obvious category of $\mathscr G$-sheaves of $\Bbbk$-vector spaces.

Note that $0(\mathscr G^0) = \{0_u\mid u\in \mathscr G^0\}$ is an open subspace of $E$, being the image of a section of the local homeomorphism $p$.  A (continuous) section $t\colon \mathscr G^0\to E$ of $p$ is said to have \emph{compact support} if \[\mathrm{supp}(t) =\{u\in \mathscr G^0\mid t(u)\neq 0_u\}\] is compact; note that it is always closed.

If $(E,p)$ is a $\mathscr G$-sheaf of $\Bbbk$-vector spaces, then we let $\Gamma_c(E)$ denote the set of  (continuous) sections of $p$ with compact support.  It is easy to see that $\Gamma_c(E)$ is a $\Bbbk$-vector space with pointwise operations (where the zero section is the additive identity).  Moreover, $\Gamma_c(E)$ has the structure of a unitary $\Bbbk \mathscr G$-module by putting
\begin{equation}\label{eq:module.action.sheaf}
(ft)(v) = \sum_{\gamma\colon u\to v} f(\gamma)\gamma t(u)
\end{equation}
 (which is a finite sum) for $f\in \Bbbk\mathscr G$ and $t\in \Gamma_c(E)$.

Since the assignment $(E,p)\mapsto \Gamma_c(E)$ is an equivalence of categories, we have that $\Gamma_c(E)$ is simple if and only if $(E,p)$ has no proper non-zero subsheaves.  A subsheaf amounts to an open subspace $E'\subseteq E$ such that $E'_u$ is a linear subspace of $E_u$ for all $u\in \mathscr G^0$ and $L_{\gamma}(E'_{s(\gamma)})\subseteq E'_{r(\gamma)}$ for all $\gamma\in \mathscr G^1$.  We say that $(E,p)$ is a \emph{simple sheaf} when the corresponding module $\Gamma_c(E)$ is simple.

If $(E,p)$ is a $\mathscr G$-sheaf of $\Bbbk$-vector spaces, we put
\[\mathrm{supp}(E) = \{u\in \mathscr G^0\mid E_u\neq \{0_u\}\}\] and note that it is in invariant subset of $\mathscr G^0$.  The following observation will be essential for us.

\begin{prop}\label{p:simple.sheaf}
Let $(E,p)$ be a simple $\mathscr G$-sheaf of $\Bbbk$-vector spaces.  Let $X=\overline{\mathrm{supp}(E)}$; note that it is a closed invariant subspace of $\mathscr G^0$.
\begin{enumerate}
\item If $u\in \mathrm{supp}(E)$, then $\overline{\orb_u}=X$.
\item  $\Bbbk\mathscr G|_{\mathscr G^0\setminus X}\subseteq \mathrm{ann}(\Gamma_c(E))$ and $\Gamma_c(E)$ is a simple $\Bbbk \mathscr G|_X$-module with the $\Bbbk\mathscr G$-action factoring through the quotient $\Bbbk\mathscr G\to \Bbbk\mathscr G|_X$.
\end{enumerate}
\end{prop}
\begin{proof}
Let $u\in \mathrm{supp}(E)$ and put $U = \mathscr G^0\setminus \overline{\orb_u}$.  Then $U$ is an open invariant subspace and so it follows immediately that \[E' = p^{-1}(U)\cup 0(\mathscr G)=\bigcup_{v\in U}E_v\cup \bigcup_{w\in \overline{\orb_u}}\{0_w\}\] is an open subsheaf.  Moreover,  since $\{0_u\}=E'_u\leq E_u\neq 0$ we have that $E'$ is a proper subsheaf and hence is the zero subsheaf by simplicity. Thus $E_v=\{0_v\}$ for all $v\in \mathscr G^0\setminus \overline{\orb_u}$.  If follows that $\mathrm{supp}(E)\subseteq \overline{\orb_u}$, which implies the first item.

For the second item, notice that if $v\notin X$, then $E_v=\{0_v\}$  and so $t(v)=0$ for all $t\in \Gamma_c(E)$.  But if $f\in \Bbbk\mathscr G|_{\mathscr G^0\setminus X}$, then by \eqref{eq:module.action.sheaf}, we see that $fs$ has support contained in $\mathscr G^0\setminus X$ for all $s\in \Gamma_c(E)$.  We conclude that $\Bbbk\mathscr G|_{\mathscr G^0\setminus X}$ annihilates $\Gamma_c(E)$.   Since $\Bbbk\mathscr G|_{\mathscr G^0\setminus X}$ is the kernel of the projection to $\Bbbk\mathscr G|_X$, the final statement follows.
\end{proof}

\begin{remark}
The $\mathscr G|_X$-sheaf corresponding to the $\Bbbk\mathscr G|_X$-module structure on $\Gamma_c(E)$ is given by restricting $(E,p)$ to $(p^{-1}(X),p)$.
\end{remark}

\section{CCR characterization}\label{sec:CCRcharacterization}
Let $\Bbbk$ be a field and let $\mathscr G$ be an ample groupoid. %Denote by $\Bbbk \G$ the Steinberg algebra associated with $\G$ over $\Bbbk$.
% We use $S$ for the inverse semigroup of compact open bisections of $\mathscr G$.
 Our goal is to characterize when $\Bbbk \G$ is CCR in terms of the isotropy groups and the topology of the orbit space in the case that $\mathscr G$ is second-countable.  %This is an exact analogue of the results in the operator case\footnote{Need to say more}.
 This is an analogue of results characterizing CCR groupoid $C^*$-algebras. In the $C^*$-algebra case the first and third authors show in~\cite[Theorem 6.1]{Cl07} and~\cite[Theorem 3.6]{vW18a}, respectively, that the groupoid $C^*$-algebra of a second-countable locally compact Hausdorff groupoid is CCR if and only if the orbit space is $T_1$ and the isotropy groups are CCR. Unlike the $C^*$-case we do not assume that $\mathscr G$ is Hausdorff, only the unit space.

Note that if $\mathscr G$ is second-countable and $\Bbbk$ is an uncountable algebraically closed field, then since $\Bbbk \mathscr G$ has countable dimension over $\Bbbk$, we have that it is CCR if and only if all its irreducible representations are by finite rank operators over $\Bbbk$.

%Let $\G$ be second-countable ample groupoid with $\G^0$ Hausdorff and let $\cG$ be the Steinberg algebra of $\G$ over $\mathbb{C}$. Let $\G(u)$ denote the stabilizer of $\G$ at $u\in\G^0$. We say $\G(u)$ is CCR is the group algebra $\mathbb{C}\G(u)$ is CCR.\footnote{
%	If we view $\G(u)$ as a groupoid in the relative topology, does the Steinberg algebra of $\G(u)$ correspond to the usual group algebra? I think so, since $\G(u)$ will be discrete and thus the steinberg alg of $\G(u)$ can be identified with the vector space of all functions with finite support, i.e. the group algebra.}

The following is our first main result of the paper.

\begin{thm}\label{t:CCRthm}
Let $\Bbbk$ be a field and $\mathscr G$ a second-countable ample groupoid.  Then the Steinberg algebra $\Bbbk \G$ is CCR if and only if the orbit space $\G^0/\G$  is $T_1$ and  $\Bbbk G_u$ is  CCR for each isotropy group $G_u$.  In particular, if $\G$ has virtually abelian isotropy groups, then $\mathbb C\G$ is CCR if and only if $\G^0/\G$ is $T_1$.
\end{thm}
\begin{proof}
Begin by assuming that $\Bbbk\G$ is CCR. First we show  that $\G^0/\G$ is $T_1$. Consider an orbit $\orb_u$ and the corresponding simple $\Bbbk\mathscr G$-module $\Bbbk\orb_u$ induced from the trivial representation of $G_u$.  In light of Lemma~\ref{l:endo.trivial}, we must have that each element of $\Bbbk\mathscr G$ acts on $\Bbbk\orb_u$ as a finite rank operator over $\Bbbk$.  Let $v\in \mathscr G^0\setminus \orb_u$ and let $U$ be a compact open neighborhood of $v$.  Then  a basis for $1_U\Bbbk \orb_u$ is $U\cap \orb_u$ and hence this set must be finite.  But $\mathscr G^0$ is Hausdorff and $v\notin \orb_u$ so we can find a smaller neighborhood $V\subseteq U$ of $v$ with $V\cap \orb_u=\emptyset$.  Thus $\orb_u$ is closed.  This shows that $\G^0/\G$ is $T_1$.

	Next we show that if $\Bbbk\G$ is CCR, then the isotropy group rings $\Bbbk G_u$ are CCR.  First note that  since $\orb_u$ is closed, it is discrete in the relative topology by Lemma~\ref{lem:discreteorbits}.  Then there is a surjective homomorphism $\Bbbk\mathscr G\to \Bbbk \mathscr G|_{\orb_u}$ because $\orb_u$ is closed.  Hence $\Bbbk \mathscr G|_{\orb_u}$ is CCR (as CCR is closed under taking quotient rings).  But $\mathscr G|_{\orb_u}$ is a discrete transitive groupoid and so $\Bbbk G_u$ is CCR by Proposition~\ref{p:discrete.groupoid}.

	Lastly, we prove the converse. Assume that $\G^0/\G$ is $T_1$ and that  $\Bbbk G_u$ is CCR for every $u\in\G^0$. We show that $\Bbbk\G$ is CCR. Let $V$ be a simple $\Bbbk\mathscr G$-module.  Then we may assume without loss of generality that $V=\Gamma_c(E)$ where $(E,p)$ is a simple $\mathscr G$-sheaf of $\Bbbk$-vector spaces.  Let $u\in \mathrm{supp}(E)$.  Because $\orb_u$ is closed it follows from Proposition~\ref{p:simple.sheaf} that $\overline{\mathrm{supp}(E)} = \overline{\orb_u}=\orb_u$.  Hence $V$ is a $\Bbbk \mathscr G|_{\orb_u}$-module and the $\Bbbk\mathscr G$-action factors through the quotient map by Proposition~\ref{p:simple.sheaf}.   Now $\mathscr G|_{\orb_u}$ is a discrete transitive groupoid by Lemma~\ref{lem:discreteorbits} and hence $\Bbbk\mathscr G|_{\orb_u}$ is CCR by Proposition~\ref{p:discrete.groupoid} and the hypothesis on $\Bbbk G_u$.  Thus $\Bbbk\mathscr G/\mathrm{ann}(V)\cong \Bbbk\mathscr G|_{\orb_u}/\mathrm{ann}(V)$  is simple with a minimal left ideal.  We conclude that $\Bbbk\mathscr G$ is CCR.
\end{proof}

%\begin{cor}
%Let $\mathscr G$ be a second-countable ample groupoid with $T_1$ orbit space and virtually abelian isotropy groups.  Then each irreducible representation of $\mathbb C\G$ is by finite rank operators over $\mathbb C$.
%\end{cor}

\section{GCR characterization}\label{sec:GCRcharacterization}

%\section{Orbit continuity}
Let $\Bbbk$ be a field and $\mathscr G$ an ample groupoid.  The set $\prim(\Bbbk \mathscr G)$ of primitive ideals of $\Bbbk \mathscr G$ is a $T_0$ topological space with the hull-kernel (or Jacobson) topology.  The basic neighborhoods in this topology are of the form \[D(f)=\{I\in \prim(\Bbbk \mathscr G)\mid f\notin I\}\] where $f\in \Bbbk\mathscr G$.

\begin{lem}\label{l:orbit.cont}
Let $u\in \mathscr G^0$ and let $\orb_u$ denote the orbit of $u$. The map $\Phi\colon \mathscr G^0\to \prim(\Bbbk \mathscr G)$ given by
\[\Phi(u) = \mathrm{ann}(\Bbbk \orb_u)\]
 is continuous and hence induces a continuous mapping $\varphi\colon \G^0/\G\to \prim(\Bbbk\G)$.	
\end{lem}
\begin{proof}
The map $\Phi$ factors through $\mathscr G^0/\mathscr G$.  Our goal is to establish continuity of $\Phi$, and hence of the induced map $\varphi\colon \mathscr G^0/\mathscr G\to \prim(\Bbbk \mathscr G)$.  To do this, we must show that the preimage of a closed set is closed.  Thus we need to show that if $u_n\to u$ is a convergent net and $f\in \mathrm{ann}(\Bbbk \orb_{u_n})$ for all $n$, then $f\in \mathrm{ann}(\Bbbk \orb_u)$.

 Let us assume that $f\notin \mathrm{ann}(\Bbbk \orb_u)$ and derive a contradiction.  Then there exists $v\in \orb_u$ such that $fv\neq 0$.  Since $\Bbbk\orb_u$ is simple, we can find $g\in \Bbbk \mathscr G$ with $gu=v$ and hence $(f\ast g)u\neq 0$.  Note that $f\ast g\in \mathrm{ann}(\Bbbk \orb_{u_n})$ for all $n$ and so replacing $f$ by $f\ast g$, we may assume without loss of generality that $fu\neq 0$ (and we now reset the notation $v$).  Then
 \[fu = \sum_{w\in \orb_u}\sum_{\gamma\colon u\to w} f(\gamma)w.\]

 Since $0\neq fu$, there exists $v\in \orb_u$ with \[\sum_{\gamma\colon u\to v} f(\gamma)\neq 0.\]  Write $f=\sum_{i=1}^r c_i1_{U_i}$ with $c_i\in \Bbbk$ and the $U_i$ compact open bisections.  In particular, for each $i$, there is at most one arrow with source $u$ in $U_i$.  Thus,
since $\mathscr G^0$ is Hausdorff, we can find a compact open neighborhood $V\subseteq\mathscr G^0$ of $v$ so that no arrow of any $U_i$ with source $u$  ends at an element of $V$ other than $v$.    Then we have
\begin{align*}
1_V\ast f &= \sum_{i=1}^rc_i1_{VU_i},\\
(1_V\ast f)u  &= \sum_{\gamma\colon u\to v}f(\gamma)v\neq 0
\end{align*}
  and $1_V\ast f\in \mathrm{ann}(\Bbbk \orb_{u_n})$ for all $n$.  Thus replacing $f$ by $1_V\ast f$, we may assume without loss of generality that $fu\neq 0$ and $f=\sum_{i=1}^r c_i1_{U_i}$ with the $U_i$ compact open bisections such that any arrow $\gamma_i\in U_i$ with source $u$ ends at same vertex $v\in \orb_u$, independent of $i$.  In particular, we have \[0\neq fu =\sum_{\gamma\colon u\to v}f(\gamma)v.\]

 Since the source map is open, if $u\in s(U_i)$, then $u_n\in s(U_i)$ for $n$ sufficiently large. Also $s(U_i)$ is compact and hence closed (since $\mathscr G^0$ is Hausdorff).   Thus if $u\notin s(U_i)$, then $u_n\notin s(U_i)$ for $n$ large enough.  Since there are only finitely many $U_i$, we can choose $n$ sufficiently large so that, for $i=1,\ldots, r$, we have that $u_n\in s(U_i)$ if and only if $u\in s(U_i)$.    Moreover, since $U_i$ is a bisection if $u_n\in s(U_i)$, then there is a unique $\gamma_{i,n}\in U_i$ with $s(\gamma_{i,n})=u_n$.

 We now compute
 \begin{equation}\label{eq:1}
 \begin{split}
 0&=fu_n\\
 &=\sum_{w\in \orb_{u_n}} \sum_{\gamma\colon u_n\to w} f(\gamma)w\\
  &=\sum_{w\in \orb_{u_n}}\sum_{\{i\mid u_n\in s(U_i), r(\gamma_{i,n})=w\}}c_iw.
\end{split}
\end{equation}
But then
\[0\neq fu =\sum_{\{i\mid u\in s(U_i)\}}c_iv =   \left(\sum_{w\in \orb_{u_n}}\sum_{\{i\mid u_n\in s(U_i), r(\gamma_{i,n})=w\}}c_i\right)v=0\]
by choice of $n$ and \eqref{eq:1} (since $u\in s(U_i)$ if and only if $u_n\in s(U_i)$, in which case there is a unique $\gamma_i\in U_i$ with $s(\gamma_i)=u$ and a unique $\gamma_{i,n}\in U_i$ with $s(\gamma_{i,n})=u_n$, and moreover $r(\gamma_i)=v$).  This contradiction completes the proof of the continuity of $\Phi$, and hence $\varphi$.
\end{proof}

The following is our second main theorem in the paper. Similarly to the CCR case,  it characterizes GCR Steinberg algebras in terms of the topology on the orbit space and the isotropy groups, and is also an analogue of a  GCR characterization for groupoid $C^*$-algebras
% In the $C^*$-algebra case the first and third authors show in
(see~\cite[Theorem 7.1]{Cl07},~\cite[Theorem 4.3]{vW18}).
%, respectively, that the groupoid $C^*$-algebra of a second-countable locally compact Hausdorff groupoid is GCR if and only if the isotropy groups are GCR and the orbit pace is $T_0$.

\begin{thm}\label{t:GCRthm}
Let $\Bbbk$ be a field and $\mathscr G$ a second-countable ample groupoid.  The Steinberg algebra $\Bbbk\G$ is GCR if and only if the orbit space $\G^0/\G$ is $T_0$ and the isotropy group rings $\Bbbk G_u$ with $u\in \mathscr G^0$ are GCR. In particular, if $\G$ has virtually abelian isotropy groups, then $\mathbb C\G$ is GCR if and only if $\mathscr G^0/\G$ is $T_0$.
\end{thm}
\begin{proof}
Assume first that $\Bbbk\G$ is GCR.  We claim that the continuous mapping $\varphi\colon \mathscr G^0/\G\to \prim(\Bbbk\G)$ sending $\orb_u$ to $\mathrm{ann}(\Bbbk \orb_u)$ from Lemma~\ref{l:orbit.cont} is injective.  By Proposition~\ref{p:gcr.spectrum} it is enough to show that if $\orb_u\neq \orb_v$, then $\Bbbk\orb_u\ncong \Bbbk \orb_v$.  Indeed, suppose $\Psi\colon \Bbbk\orb_u\to \Bbbk \orb_v$ is an isomorphism.  Let $\Psi(u) = \sum_{w\in \orb_v} c_ww$ with the sum finite.  Then since $u\notin \orb_v$ and $\mathscr G^0$ is Hausdorff, we can find a compact open neighborhood $U$ of $u$ containing  no element of the finite support of $\Psi(u)$.  Then $\Psi(u) = \Psi(1_Uu)=1_U\Psi(u) = \sum_{w\in \orb_v}c_w1_Uw=0$, a contradiction.
Since $\prim(\Bbbk\G)$ is $T_0$ and $\mathscr G^0/\G$ embeds continuously in it, we conclude that $\G^0/\G$ is $T_0$.

Next we verify that $\Bbbk G_u$ is GCR.  Note that $\Bbbk \mathscr G|_{\overline{\orb_u}}$ is a quotient of $\Bbbk\G$ and hence is GCR.  By Proposition~\ref{p:locally.closed}, we have that $\mathscr G|_{\orb_u}$ is open in $\mathscr G|_{\overline{\orb_u}}$ (in the relative topology) and so $\mathscr G|_{\orb_u}$ is discrete in the relative topology by Lemma~\ref{lem:discreteorbits}.  Also $\Bbbk\mathscr G|_{\orb_u}$ is an ideal in $\Bbbk\mathscr G|_{\overline{\orb_u}}$ with local units (as $\orb_u$ is open in $\overline{\orb_u}$) and hence is GCR by Corollary~\ref{c:ideal.inherit}.  Thus, since $\mathscr G|_{\orb_u}$ is discrete and transitive, we deduce that $\Bbbk G_u$ is GCR by Proposition~\ref{p:discrete.groupoid}.

Suppose now that $\mathscr G^0/\G$ is $T_0$ and each isotropy group ring is GCR and let $V$ be a simple $\Bbbk\G$-module. Put $D=\mathrm{End}_{\Bbbk\G}(V)$.  We may assume without loss of generality that $V=\Gamma_c(E)$ with $(E,p)$ a simple $\mathscr G$-sheaf of $\Bbbk$-vector spaces.  Let $u\in \mathrm{supp}(E)$.  Then $V$ is a $\Bbbk \mathscr G|_{\overline{\orb_u}}$-module and the $\Bbbk\mathscr G$-action factors through the quotient map by Proposition~\ref{p:simple.sheaf}.  By Proposition~\ref{p:locally.closed}, $\mathscr G|_{\orb_u}$ is an open subgroupoid of $\mathscr G|_{\overline{\orb_u}}$ in the relative topology, and hence  $\mathscr G|_{\orb_u}$ is a discrete transitive groupoid by Lemma~\ref{lem:discreteorbits} in the relative topology.  Thus $\Bbbk\mathscr G|_{\orb_u}$ is GCR by Proposition~\ref{p:discrete.groupoid}.  Moreover, it is an ideal in $\Bbbk \mathscr G|_{\overline{\orb_u}}$ and is a ring with local units. Since $u\in \mathrm{supp}(E)$, there exists $t\in \Gamma_c(E)$ with $t(u)\neq 0_u$. Indeed, take $0_u\neq e\in E_u$ and choose $W$ a compact open neighborhood of $e$ with $p|_W$ a homeomorphism; then extending $(p|_W)^{-1}$ to be zero outside $p(W)$ gives a section $t\in \Gamma_c(E)$ with $t(u)=e\neq 0_u$.  Since $\orb_u$ is open in $\overline{\orb_u}$, we can find a compact open neighborhood $U$ of $u$ with $U\cap \overline{\orb_u}\subseteq \orb_u$.  Note that $1_Ut = t|_U\neq 0$ as $t(u)\neq 0_u$.  Thus $1_{U\cap \overline{\orb_u}}\in \Bbbk G|_{\orb_u}$ does not annihilate $V$, viewed as a $\Bbbk \mathscr G|_{\overline{\orb_u}}$-module.  We conclude that $V$ is a simple $\Bbbk \mathscr G|_{\orb_u}$-module with $\mathrm{End}_{\Bbbk \mathscr G|_{\orb_u}}(V)=D$ by Lemma~\ref{l:restrict.ideals}.  Thus some element of $\Bbbk \mathscr G|_{\orb_u}$ acts on $V$ as a non-zero element of finite rank over $D$ and hence some element of $\Bbbk \G$ acts on $V$ as a non-zero element finite rank over $D$ (as the action on $V$ factors through the surjective homomorphism to $\Bbbk\mathscr G|_{\overline{\orb_u}}$, which contains $\Bbbk \G|_{\orb_u}$).  This completes the proof that $\Bbbk \G$ is GCR.
\end{proof}

%%%%%%%%%%%%%%%%%%%%%%%%%%%%%%%%%%%%%%%%%%%%%%%%%%%%%%%%
\section{CCR and GCR higher-rank graph algebras}\label{sec:higerRankgraphs}
Countable finitely aligned higher-rank graphs include all (countable) directed graphs as well as row-finite and locally convex (countable) higher-rank graphs.  They were introduced in~\cite{RSY04} where  the authors study higher-rank graph $C^*$-algebras.
In~\cite{CP17}, the authors introduced Kumjian-Pask algebras of finitely aligned higher-rank graphs, a purely algebraic analogue of the higher-rank graph $C^*$-algebras.   Given a finitely aligned higher-rank graph $\Lambda$, they show that the Kumjian-Pask algebra KP$_{\Bbbk}(\Lambda)$ is isomorphic to the Steinberg algebra
$\Bbbk\G_{\Lambda}$ where $\G_{\Lambda}$ is
the \emph{boundary path groupoid} (from~\cite{Y07}).  We adopt the notation of~\cite{CP17} and refer the reader there for more details.

The unit space of $\G_{\Lambda}$ is the boundary path space $\partial \Lambda$.  A base for the topology on $\partial \Lambda$ is given by the \emph{cylinder sets}:
For $\mu \in \Lambda$ and finite $F \subseteq s(\mu)\Lambda$ we define
\[Z(\mu) : = \{\mu x \in \partial\Lambda\}  \quad \text{ and } \quad Z(\mu \setminus F) := Z(\mu) \setminus \bigcup_{\nu \in F}Z(\mu\nu).\]

For each $x \in \partial\Lambda$, $x$ is a path and $\orb_x$ is the collection of paths in $\partial\Lambda$ that are eventually the same as $x$, that is, $y \in \orb_x$ if and only if there exists
$p,q \in \mathbb{N}^k$ such that $\sigma^p(x)=\sigma^q(y)$ where $\sigma$ is the shift map.

\begin{prop}
\label{p:condM}
Let $\Lambda$ be a countable finitely aligned higher-rank graph  and $\G_{\Lambda}$ the boundary path groupoid.  Then $\G^{0}_{\Lambda}/\G_{\Lambda}$ is $T_1$ if and only if the following condition is satisfied:
\begin{equation}
\label{condM}	
\text{for each $x \in \partial\Lambda$ and $v \in \Lambda^0$ we have $Z(v)\cap \orb_x$ is finite.}
\end{equation}
\end{prop}
\begin{proof}
First suppose $\G^{0}_{\Lambda}/\G_{\Lambda}$ is $T_1$. Fix $x \in \partial\Lambda$ and $v \in \Lambda^0$.    Since $\G^{0}_{\Lambda}/\G_{\Lambda}$ is also $T_0$, Proposition~\ref{p:locally.closed} says $\orb_x$ is discrete in the relative topology.  By way of contradiction, suppose $Z(v) \cap \orb_x$ is infinite.  Then there exists a sequence $(x_n) \subseteq Z(v) \cap \orb_x$ where each term is distinct.
 Further, since $\orb_x$ is closed, $Z(v) \cap \orb_x$ is compact so $(x_n)$ has a subsequence that converges to $z \in Z(v) \cap \orb_x$.  But $\{z\}$ is open in $\orb_x$ so any sequence in $\orb_x$ converging to $z$ must eventually be the constant sequence $z,z,...$ which is a contradiction.

For the reverse implication, fix $x \in \partial\Lambda$.  To see that  $\G^{0}_{\Lambda}/\G_{\Lambda}$ is $T_1$, we show $\orb_x$ is closed in $\G^0$.  Suppose $(y_n) \subseteq \orb_x$ such that $y_n$ converges to $y$.  So $y_n$ is in the finite set $ Z(r(y)) \cap \orb_x$ eventually and hence the sequence is eventually constant.  Thus $y=y_n$ eventually and $\orb_x$ is closed.
\end{proof}

We are now in a position to characterize when the algebras associated to finitely aligned higher-rank graphs are CCR.  Since $\G_{\Lambda}$ is second countable and all of its isotropy groups are abelian (as they are subgroups of $\mathbb{Z}^k$), we combine Proposition~\ref{p:condM} with Theorem~\ref{t:CCRthm}, Corollary~\ref{c:CCR.group} and~\cite[Theorem~6.1]{Cl07} to get the following corollary.  This gives an alternate proof to Ephrem's directed graph $C^{\ast}$-algebra result~\cite[Theorem~5.5]{Eph04}, where he calls condition \eqref{condM} `condition M'.

\begin{cor}\label{c:ccrgraph}
 Let $\Lambda$ be a countable finitely aligned higher-rank graph and $\Bbbk$ an uncountable algebraically closed field.  Then the following are equivalent.
 \begin{enumerate}
  \item condition \eqref{condM} is satisfied;
  \item KP$_{\Bbbk}(\Lambda)$ is CCR;
  \item $C^{\ast}(\Lambda)$ is CCR.
 \end{enumerate}
\end{cor}

Since Leavitt path algebras are special cases of Kumjian-Pask algebras, Corollary~\ref{c:ccrgraph}, in particular, characterizes the CCR Leavitt path algebras.

Next we characterize the GCR higher rank graph algebras.  Since $\G^{0}_{\Lambda}/\G_{\Lambda}$ is $T_0$ if and only if orbits are discrete when endowed with the subspace topology, the  following proposition is immediate.
\begin{prop}
\label{p:condN}
Let $\Lambda$ be a countable finitely aligned higher-rank graph  and $\G_{\Lambda}$ the boundary path groupoid.  Then
$\G^{0}_{\Lambda}/\G_{\Lambda}$ is $T_0$ if and only if the following condition is satisfied
\begin{align}&\label{condN}\text{for each $x \in \partial\Lambda$, there exists $\mu \in \Lambda$ and finite $F \subseteq s(\mu)\Lambda$ such that }\\
&Z(\mu \setminus F) \cap \orb_x = \{x\}.\notag
\end{align}
\end{prop}

If we restrict our attention to row-finite higher-rank graphs without sources, one can show that condition ~\eqref{condN} is equivalent to saying for each infinite path $x$, there exists $n \in \mathbb{N}^k$ such that for $v:=r(\sigma^n(x))$,  the only paths from $v$ to $x$ are the paths on $x$ itself, stated more precisely, this means:
\[\text{if $\lambda \in \Lambda$ with $s(\lambda)=v$ and $r(\lambda)$ on $x$, then
 $\lambda=\sigma^n(x)(0, d(\lambda))$.}
\]
This condition looks more like Ephrem's GCR condition in~\cite[Theorem~7.3]{Eph04} and is, in fact, equivalent to his in the row-finite directed graph setting.  Combining our Proposition~\ref{p:condN}, Theorem~\ref{t:GCRthm} and Corollary~\ref{c:CCR.group} with~\cite[Theorem~7.1]{Cl07} we get our final result.

\begin{cor}\label{c:gcrgraph}
 Let $\Lambda$ be a countable finitely aligned higher-rank graph and $\Bbbk$ an uncountable algebraically closed field.  Then the following are equivalent.
 \begin{enumerate}
  \item condition \eqref{condN} is satisfied;
  \item KP$_{\Bbbk}(\Lambda)$ is GCR;
  \item $C^{\ast}(\Lambda)$ is GCR.
 \end{enumerate}
\end{cor}

Corollary~\ref{c:gcrgraph} characterizes, in particular, all the GCR Leavitt path algebras.

\begin{ex} \label{ex:graph} If $\G$ is a discrete groupoid, then the orbit space is discrete and hence $T_1$.  Here we give an example of a boundary path groupoid that has non-discrete $T_1$ orbit space.    Let $E$ be the full rooted binary tree with edges directed as follows (using the so-called \emph{southern convention}):

\bigskip
  
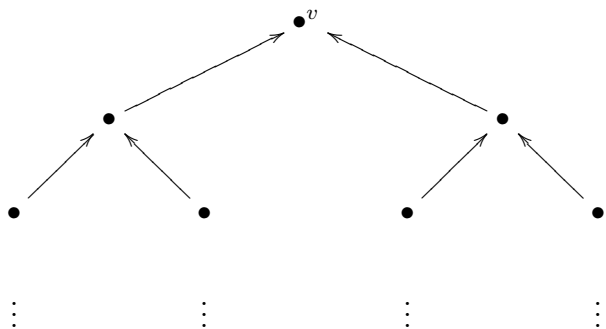
\[ \xymatrix{
&&&\bullet^v&&&  \\
& \bullet \ar[urr] &&&&  \bullet \ar[ull] &   \\
\bullet \ar[ur] & & \bullet \ar[ul] &&   \bullet \ar[ur] && \bullet \ar[ul] \\
\vdots && \vdots &&\vdots&&\vdots
}
\]

\bigskip

\noindent Then $E$ is a row-finite graph with no sources and the associated boundary path space is the infinite path space, that is, the space of all infinite sequences of edges $e_1e_2 \dots$ such that $s(e_i) = r(e_{i+1})$.  Here, each orbit contains a unique infinite path whose range is the root $v$ and one can show the orbit space is homeomorphic to the Cantor set which is compact and Hausdorff hence $T_1$.
\end{ex}

\subsection*{Acknowledgments}
We are grateful to the anonymous referees for their careful reading of the paper. In particular, one of the referees provided the arguments for Lemma~\ref{l:restrict.ideals} at its full level of generality and suggested the current version of Corollary~\ref{c:ideal.inherit}.

%\bibliographystyle{abbrv}
%\bibliography{GCRbibl}
%
%\bibliographystyle{acm}
%\begin{thebibliography}{10}
%	
%\bibitem{Bre14}
%	{\sc Bre\v{s}ar, M.},
%	\newblock {\em Introduction to noncommutative algebra}
%	\newblock  Springer, 2014
%
%
%\end{thebibliography}
%	
\end{document}